\documentclass[12pt]{article}

\usepackage{xpatch} 

\usepackage{amssymb,amsfonts, amsmath, amsthm, dutchcal, mathtools}
\usepackage{hyperref}
\usepackage{cleveref}
\usepackage{xcolor} 
\usepackage[normalem]{ulem} 
\usepackage{enumerate}

\usepackage{tikz} 
\usetikzlibrary{cd} 

\def\N{\mathbb{N}}
\def\R{\mathbb{R}}
\def\C{\mathbb{C}}
\def\L{\mathbb{L}}

\def\l{\lambda}
\def\Re{\mathrm{Re}}
\def\Im{\mathrm{Im}}

\def\eps{\varepsilon}
\def\phi{\varphi}
\def\Eps{\mathcal{E}}
\def\0{{\bf 0 }}
\def\ol{\overline}
\def\Span{\textrm{span}}

\newcommand{\dis}[1]{\displaystyle{#1}} 
\newcommand{\norm}[1]{\left\Vert #1 \right\Vert}
\newcommand{\abs}[1]{\left\vert #1 \right\vert}
\newcommand{\ip}[2]{\left\langle #1, #2 \right\rangle}

\def\idempotents{\mathbb{P}}
\renewcommand\P{\mathbb{P}}

\newcommand{\comment}[1]{} 

\newcommand{\defeq}{\ensuremath{\mathop{:}\!\!=}}

\newtheorem{theorem}{Theorem}[subsection]
\newtheorem{lemma}[theorem]{Lemma}
\newtheorem{corol}[theorem]{Corollary}
\newtheorem{prop}[theorem]{Proposition}

\theoremstyle{definition}

\newtheorem{defn}[theorem]{Definition}
\newtheorem{example}[theorem]{Example}
\newtheorem{remark}[theorem]{Remark}

\numberwithin{equation}{section}

\begin{document}

\title{$\L$-functional analysis}

\author{Eder Kikianty, Miek Messerschmidt,  
Luan Naude, \\ Mark Roelands, Christopher Schwanke, Walt van Amstel, \\
Jan Harm van der Walt, and Marten Wortel}

\maketitle

\begin{abstract}
    Inspired by the theories of Kaplansky-Hilbert modules and probability theory in vector lattices, we generalise functional analysis by replacing the scalars $\R$ or $\C$ by a real or complex Dedekind complete unital $f$-algebra $\L$; such an algebra can be represented as a suitable space of continuous functions. We set up the basic theory of $\L$-normed and $\L$-Banach spaces and bounded operators between them, we discuss the $\L$-valued analogues of the classical $\ell^p$-spaces, and we prove the analogue of the Hahn-Banach theorem. We also discuss the basics of the theory of $\L$-Hilbert spaces, including projections onto convex subsets, the Riesz Representation theorem, and representing $\L$-Hilbert spaces as a direct sum of $\ell^2$-spaces. 
\end{abstract}

\section{Introduction}

In 1953, Kaplansky (\cite{kaplansky53}) introduced the theory of AW*-modules, generalizing the theory of Hilbert spaces, to solve a problem in AW*-algebras. An AW*-algebra is a slight generalization of a von Neumann algebra, and each commutative AW*-algebra can be represented as a $C(K)$-space where $K$ is Stonean, i.e., an extremally disconnected compact Hausdorff space. An AW*-module is a module over a commutative AW*-algebra $A$ equipped with an inner product taking values in $A$ with some additional assumptions mimicking the norm completeness property of the classical case. It turned out that the theory of these AW*-modules is remarkably similar to the theory of ordinary Hilbert spaces. 
AW*-modules were later known as Kaplansky-Hilbert modules, or KH-modules for short.

Later on, in the 1970s, the theory was extended to so called Hilbert C*-modules, in which $A$ is more general: it is a (not necessarily commutative) C*-algebra. An important aspect of the theory of Hilbert C*-modules, however, is different from the theory of KH-modules: norm convergence is used in $A$, whereas in the theory of KH-modules, order convergence is used in the commutative AW*-algebra. Hilbert C*-modules became extremely important in noncommutative geometry, but compared to KH-modules, they are further away from classical Hilbert spaces. 
For example, consider $C(K)$ (where $K$ is compact Hausdorff) as a Hilbert $C(K)$-module, let $\omega \in K$ be non-discrete, and consider the submodule $M = \{f \in C(K) \colon f(\omega) = 0\}$. Then $M^{\perp\perp} = C(K)$, but $\ol{M} = M \not= C(K)$ so that $M^{\perp\perp} \not= \ol{M}$ unlike in the classical case. If instead of norm convergence one considers order convergence, then this problem disappears, as $M$ is order dense in $C(K)$.

Very recently, Edeko, Haase, and Kreidler (\cite{edeko2023decomposition}) gave a very nice application of this theory. To briefly summarize this paper: they proved a spectral theorem for Hilbert-Schmidt operators on KH-modules using elementary means, a theorem vastly generalizing unitary group representations on KH-modules, and they combine these results to obtain a new KH-module theoretic proof of the classical and celebrated Furstenberg-Zimmer structure theorem in ergodic theory, which takes away the classical restrictions of separability and Borel spaces. They also promise more future applications in ergodic theory using the theory of KH-modules.

On the other hand, since the 2000s (see \cite{ppt, kkw} and references therein) a comprehensive theory of probability theory in vector lattices has been developed. In this theory, a probability space and its space of integrable functions $L^1(\Omega)$ is replaced by a Dedekind complete vector lattice $E$ with weak order unit, equipped with a conditional expectation: a positive order continuous operator $T \colon E \to E$ with Dedekind complete range space $R(T)$ leaving the weak order unit invariant. One can then extend $T$ to its maximal domain and it turns out that $R(T)$ becomes a universally complete vector lattice, and so it is equipped with a canonical $f$-algebra multiplication. The properties of the conditional expectation then ensure that $E$ becomes an $R(T)$-module (see \cite[Theorem~2.2]{mixing}). For $1 \leq p \leq \infty$, one can then define $L^p(T)$ as $R(T)$-modules where $T|\cdot|^p$ acts as an $R(T)$-valued norm (see \cite[Theorem~3.2]{mixing} for the cases $p=1$ and $p=\infty$). Recently a Riesz Representation Theorem for $L^2(T)$ has been proven (see \cite{kkw} and \cite{kalauch2022hahnjordan}).

Although the similarity between the above theory and the theory of KH-modules is not immediately clear, note that a commutative AW*-algebra can also be described in vector lattice terms as a Dedekind complete vector lattice with strong order unit, which is then automatically equipped with a canonical $f$-algebra multiplication. Therefore the space that acts as the scalars in both theories, i.e., a commutative AW*-algebra and $R(T)$, are both Dedekind complete unital $f$-algebras (with unital we mean the multiplicative unit, which is then automatically a weak order unit, see \cite[Theorem~10.7]{depagter}). 

Neither the strong order unit nor the universal completeness seem to be necessary to develop the general theory of Banach and Hilbert spaces taking values in such spaces, and so in this paper we attempt to simultaneously generalise the basic results of both theories. Therefore, in our approach, we take $\L$ to be a Dedekind complete unital $f$-algebra, and we develop the basic theory of $\L$-normed and $\L$-inner product spaces. We chose the letter $\L$ for lattice, so that $\L$-normed space reads as lattice-normed space\footnote{The blackboard bold is used to emphasize that we view $\L$ as the scalars (as is done in \cite{edeko2023decomposition}). We deviate from \cite{edeko2023decomposition} by writing the norm as $\norm{\cdot}$ (they use $|\cdot|$) to emphasize the similarities with the classical theory of Banach spaces.}.  Lattice-normed spaces were first introduced by Kantorovich in 1936 in \cite{kantorovich1936}. For more details on lattice normed spaces we recommend \cite[Section~2]{domops}; the same book contains KH-module theory in \cite[Section~7.4]{domops}.

Beyond the motivating connections between our theory and the theories of KH-modules and probability theory in vector lattices, we briefly remark upon some further connections with the existing literature: 

\begin{enumerate}
    \item The first connection is with random functional analysis; for a nice summary of this theory and its history, we recommend the introduction of \cite{randomFA}. The main idea of random functional analysis is to replace the scalar field with $L^0(P)$ where $P$ is a probability measure. Note that $L^0(P)$ is a universally complete vector lattice, but an important difference with our theory is that order convergence is not taken as the central convergence structure on $L^0(P)$. Instead, the topology of convergence in probability is mainly used. This topology is still somewhat related to order convergence, since it is an order continuous topology, meaning that order convergent nets convergence in this topology. (Note that the theory of Hilbert C*-modules uses norm convergence, which is not order continuous, as the convergence structure on the C*-algebra acting as the scalars.) On $L^0(P)$-normed modules, the topology generated by balls of radius $\eps$, where $\eps \in L^0(P)$ is strictly positive (see \cite[Proposition~2.6]{randomFA}) is sometimes also considered.  
    \item In the literature, the notion of an $L^0$-valued Hilbert spaces is considered; for example, \cite[Section~1]{L0valuedHK} discusses several results. These are sometimes also called `Kaplansky-Hilbert modules over $L^0$', and this terminology makes sense, as unlike in random functional analysis, order convergence is used in $L^0$, making the theory very similar to the theory of KH-modules (and therefore also to our theory). However, we have not been able to find a structured account of a theory of $L^0$-valued Hilbert spaces.
    \item Another connection with the literature is with Boolean-valued analysis, a theory containing a lot of model theory. For example, the paper \cite{takeuti} is about Boolean-valued analysis, and it contains the remark that its results can also be derived by considering $L^0$-valued Hilbert spaces. Other connections with Boolean-valued analysis can be found in \cite{booleanvaluedanalysis} and \cite[Section~8]{domops}. The paper \cite{avilesgarcia_booleanvaluedanalysis} claims that many results in so-called locally convex $L^0$-modules can be proven using Boolean valued analysis techniques, and it seems possible that this may also be the case for our theory. We have not investigated this but encourage the Boolean-valued analysis experts to do so.
    \item Finally, we would like to mention \cite{powerseries}, which in view of this paper could be described as `$\L$-valued complex analysis'. The theory in \cite{powerseries} is currently being developed further by their authors.
\end{enumerate}

Although this paper contains vector lattice theory, knowledge of this theory is not necessary; only a basic familiarity of functional analysis is required to comfortably read this paper. Indeed, we made the conscious decision to make use of the representation theory for $\L$ (see \Cref{ss:representation theory}) over vector lattice arguments so that the paper is accessible to a wider audience. The reader will find occasional remarks where we have added further context from the vector lattice point of view or an alternative proof using vector lattice techniques. 

\medskip

The paper is organized as follows:

\medskip

In \Cref{s:general information L}, we establish basic results for the Dedekind complete unital $f$-algebra $\L$ needed in the rest of the paper, including the representation of $\L$ as a certain space of continuous functions.

\Cref{s:L-normed spaces} contains the basic theory of $\L$-normed spaces, $\L$-Banach spaces, and operators between them. Beyond the notions of convergence, completeness, convexity, and quotient spaces, we also define the normalisation\footnote{First defined in \cite[Section~2.1]{edeko2023decomposition}.} of elements in an $\L$-Banach space, which is used extensively in the rest of the paper. We also introduce the familiar function spaces $c_0(S,Y)$ and $\ell^\infty(S,Y)$ for a non-empty set $S$ and an $\L$-normed space $Y$ (generalizing the important special case where $Y = \L$).

In \Cref{s:l^p spaces}, we discuss the space of $p$-summable functions $\ell^p(S,Y)$ for $1 \leq p < \infty$. We also establish the familiar duality results $c_0(S,Y)^* \cong \ell^1(S, Y^*)$ and $\ell^p(S,Y)^* \cong \ell^q(S, Y^*)$ for $1 \leq p < \infty$ with $q$ its conjugate index. 

The $\L$-valued Hahn-Banach theorem is proven in \Cref{s:hahn banach} and its consequences on the structures of dual spaces is discussed. 

Finally, in \Cref{s:hilbert}, we establish the basic theory for $\L$-inner product and $\L$-Hilbert spaces which includes projections onto convex subsets, (sub)orthonormal bases, Parseval's identity, and the $\L$-valued Riesz Representation theorem. In \Cref{ss:representation theory for L-Hilbert spaces}, we show that every $\L$-Hilbert space is isometrically isomorphic to a certain direct sum of $\ell^2$-spaces, see \Cref{c: Representation for L-Hilbert spaces}.

\section{General information about $\L$}\label{s:general information L}

\subsection{Representation theory}\label{ss:representation theory}

We first introduce $C_\infty(K)$\footnote{See \cite[Section 7.2]{AliprantisBurkinshaw2003} for details; there, and more commonly in the literature, it is denoted $C^\infty(K)$, but we prefer the notation $C_\infty(K)$ since it does not suggest any differentiability properties.}. Let $K$ be a Stonean space; recall that a Stonean space is a compact Hausdorff space so that the closure of every open set is open. The set $C_\infty(K)$ denotes the set of all continuous extended real-valued functions $\l$ on $K$ such that $\l$ is finite-valued on a dense (and by continuity open) subset $U_\l$ of $K$. For $\l, \mu \in C_\infty(K)$, the function $\l + \mu$ is defined on the open and dense subset $U_\l \cap U_\mu$ of $K$, which then uniquely extends to a function in $C_\infty(K)$; this procedure defines addition on $C_\infty(K)$. The lattice operations and (scalar) multiplication are defined similarly, and these turn $C_\infty(K)$ into a vector lattice and an algebra. 

\comment{

We will first introduce $C_\infty(K)$, see \cite[Section 7.2]{AliprantisBurkinshaw2003} for details; there, and more commonly in the literature, it is denoted $C^\infty(K)$, but we prefer the former notation since it does not suggest any differentiability properties. Let $K$ be a Stonean space; recall that a Stonean space is a compact Hausdorff space so that the closure of every open set is open. The set $C_\infty(K)$ denotes the set of all continuous extended real-valued functions $\l$ on $K$ such that $\l$ is finite-valued on a dense (and by continuity open) subset $U_\l$ of $K$. For $\l, \mu \in C_\infty(K)$, the function $\l + \mu$ is defined on the open and dense subset $U_\l \cap U_\mu$ of $K$, which then uniquely extends to a function in $C_\infty(K)$; this procedure defines addition on $C_\infty(K)$. The lattice operations and (scalar) multiplication are defined similarly, and these turn $C_\infty(K)$ into a vector lattice and an algebra. 

}

Central to our theory is replacing the scalar field $\R$ by a Dedekind complete unital $f$-algebra $\L$. To understand this, it is not necessary to have any knowledge about $f$-algebras. (We refer the readers interested in $f$-algebras to \cite{depagter}.) Indeed, combining \cite[Theorem 2.64]{AliprantisBurkinshaw2006} and \cite[Theorem 7.29]{AliprantisBurkinshaw2003} shows that there exists a unique Stonean space $K$ (the Stone space of the Boolean algebra of bands of $\L$) so that $\L$ is an order dense sublattice and subalgebra of $C_\infty(K)$ containing the constant one function $\mathbf{1}_K$ as the multiplicative unit. (Note that $C_\infty(K)$ is the universal completion of $\L$.) According to \cite[Theorem 1.40]{AliprantisBurkinshaw2003}, $\L$ is an order ideal in $C_\infty(K)$, i.e., if $\l \in \L$, $\mu \in C_\infty(K)$ and $|\mu| \leq n|\l|$ for some $n \in \N$, then $\mu \in \L$.  Therefore, since $\mathbf{1}_K\in\L$, it follows that $C(K)\subseteq \L$. 

\begin{remark}\label{rem: Representation theory for L}
Based on the above observations, we will assume when convenient that for some fixed Stonean space $K$, 
\[ C(K)\subseteq \L\subseteq C_\infty(K).\]
Here $C(K)$ is an order ideal and subalgebra of $\L$, as is $\L$ of $C_\infty(K)$. To emphasize the fact that we treat elements of $\L$ as scalars, we denote the multiplicative unit of $C(K)$, $\L$, and $C_\infty(K)$ by 
\[ 
1 \defeq \mathbf{1}_K.
\]

\end{remark}
In the sequel we will use the phrase `by representation theory' if we use \Cref{rem: Representation theory for L} to prove properties about $\L$. As an example, if $0 \leq r < 1 \in \R$ and $\l \in \L^+$, the element $\l^r$ is defined on $U_\l$ and hence extends to an element of $C_\infty(K)$. (We use the convention that $\l^0=1$.) Now $0\leq \l^r \leq \l \vee 1 \in \L$ (as this inequality holds pointwise on $U_\lambda$) and since $\L$ is an order ideal of $C_\infty(K)$, we have $\l^r \in \L$. This allows us to define, for $0 \leq p <\infty$ and $\l \in \L^+$, the element $\l^p \in \L$ as $\l^n \l^r$ where $n \in \N_0$ and $0 \leq r < 1$ with $p = n + r$.

\begin{remark}
We acknowledge that non-negative powers in $\L$ can be defined (1) without the use of representation theory, and (2) in a wider class of spaces, called \textit{weighted geometric mean closed} unital $f$-algebras. The reader is referred to \cite[Section~4]{BusSch3} for details. Another advantage to using the definition given in \cite[Section~4]{BusSch3} is that it requires only the (less controversial) countable axiom of choice. However, defining non-negative powers via representation theory proves to be more convenient for the purposes of this paper.
\end{remark}

\begin{remark}\label{r:real or complex}
Our real $f$-algebra $\L$ can be complexified to a complex vector space $\L_\C:=\L+i\L$ which is a  `complex vector lattice' (despite its name, this is actually not a lattice!), meaning that its real part is $\L$ and it is equipped with a natural modulus $| \cdot | \colon \L_\C \to \L^+$ extending the modulus (absolute value) of $\L$ still satisfying the crucial properties: $|\l|=0$ if and only if $\l=0$, $|\l + \mu| \leq |\l| + |\mu|$ and $|\l \mu| = |\l||\mu|$ for all $\l, \mu \in \L_\C$. In our case, when considering $C(K) \subseteq \L \subseteq C_\infty(K)$, the complex vector lattice $\L_\C$ naturally satisfies $C(K)_\C \subseteq \L_\C \subseteq C_\infty(K)_\C$ where the modulus $|\l|$ is given by 
\begin{equation}\label{e:defn-modulus}
\sqrt{(\Re \l)^2 + (\Im \l)^2} = \sup_{\zeta\in\mathbb{T}} \Re(\zeta \l).
\end{equation}
The identity \eqref{e:defn-modulus} is clear from representation theory. It was shown for square mean closed semiprime $f$-algebras in \cite[Theorem 2.23]{azouzi}. For more information regarding complex vector lattices, see \cite{deSchipper}.

In results to follow, the scalar field over which $\L$ is defined will almost always be immaterial to the arguments. We will use the adjectives `real' and `complex' before $\L$ in the cases where this is required.
\end{remark}

The next remark is of general interest but we shall not use it.

\begin{remark}\label{r:hyperstonean}
Note that in \Cref{rem: Representation theory for L}, the order continuous real-valued functionals on $C(K)$ separate the points of $C(K)$ if and only if $K$ is \emph{hyper-Stonean} (i.e. the union of the supports of the normal measures on $K$ is dense in $K$), if and only if $C(K)_\C$ is a von Neumann algebra, if and only if $C(K)$ can be represented as a space $L^\infty$ of essentially bounded measurable functions with respect to a decomposable measure on a locally compact Hausdorff space, see for instance \cite[Theorem 6.4.1]{DalesDashiellLauStrass2016}. In this case $C_\infty(K)$ can be identified with the space of all measurable functions $L^0$, so that $L^\infty \subseteq \L \subseteq L^0$.
\end{remark}

\subsection{Basic properties of $\L$}

The results in this subsection are well-known in the vector lattice community but have been added for the sake of completeness. The usual upper bound juggling yields the first lemma.

\begin{lemma} \label{l:sup_additive}
For real $\L$, let $A, B \subseteq \L$ be non-empty and bounded above. Then 
\[ \sup(A+B) = \sup(A) + \sup(B). \]
Similarly, if $A$ and $B$ are non-empty and bounded below, then
\[ \inf(A + B) = \inf(A) + \inf(B). \]
\end{lemma}

\comment{
\begin{proof}
It is clear that $\sup(A) + \sup(B)$ is an upper bound for $A + B$. Let $u$ be any upper bound of $A+B$. For fixed $b\in B$, we have $a \leq u - b$ for all $a\in A$. Thus, for all $b\in B$, $\sup\left( A \right) \leq u - b$, hence $b \leq u - \sup\left( A \right)$, and so $ \sup\left( B \right) \leq u-\sup\left( A \right)$. Therefore $\sup(A) + \sup(B) \leq u$ which proves that $\sup\left( A + B \right) = \sup(A) + \sup(B)$. The statement for infima can be proven similarly. 
\end{proof}
}

\begin{lemma}\label{l:invertible}
   Let $\l \in \L^+$ with $\l$ invertible, then $\l^{-1} \in \L^+$, and if $\mu \geq \l$, then $\mu$ is invertible and $\mu^{-1} \leq \l^{-1}$. Furthermore, if $\L$ is complex and $\l + i\mu \in \L$, then $\l + i\mu$ is invertible if either $\l$ or $\mu$ is invertible.
\end{lemma}

\begin{proof}
By representation theory, $\l^{-1} \in \L^+$ and the existence of $\mu^{-1} \in C_\infty(K)$ with $0 \leq \mu^{-1} \leq \l^{-1}$ is clear. Now $\mu^{-1} \in \L$ because $\L$ is an order ideal of $C_\infty(K)$. The final statement follows from $(\l + i \mu)^{-1} = (\l - i\mu)(\l^2 + \mu^2)^{-1}$.
\end{proof}
For a proof of \Cref{l:invertible} without using representation theory, see \cite[Theorems~3.6(ii)~and~11.1]{depagter}.

As stated in \cite[Theorem~2.62]{AliprantisBurkinshaw2006}, for every $\mu\in\L$, the multiplication operator $\lambda\to \mu\lambda$ on $\L$ is an orthomorphism. Hence, by \cite[Theorem~2.44]{AliprantisBurkinshaw2006}, multiplication by a fixed element in $\L$ is order continuous. Lemma~\ref{l:order_cont_mult} below is a direct consequence of these two facts. However, there exists an easy proof of this result, which we include here for the convenience of the reader.

\begin{lemma}\label{l:order_cont_mult} 
For real $\L$, let $A \subseteq \L$ be non-empty and bounded above with $\l \in \L^+$. Then 
\[
        \sup\left(\l A \right) = \l \sup\left(A\right). 
\]Similarly, if $A$ is non-empty and bounded below, then 
\[
        \inf\left(\l A \right) = \l \inf\left(A\right).
\]
\end{lemma}

\begin{proof}
    We prove the statement for suprema. Since $\l a \leq \l \sup(A)$ for all $a \in A$, it follows that $\sup(\l A) \leq \l \sup( A )$. For the reverse inequality, if $a \in A$, we have
    \[ 
    (\l + 1)a = \l a + a \leq \sup(\l A) + \sup(A).
    \]
    Since $1 \leq \l + 1$, it follows by \Cref{l:invertible} that $\l + 1$ is invertible and $(\l + 1)^{-1} \in \L^+$. Hence
    \[ 
    a \leq (\l + 1)^{-1} (\sup(\l A) + \sup(A)) 
    \]
    and so $\sup A \leq (\l + 1)^{-1} (\sup(\l A) + \sup(A))$, implying that 
    \[
    \l \sup(A) + \sup(A) = (\l + 1) \sup(A) \leq \sup(\l A) + \sup(A).
    \]
    Therefore $\l \sup(A) \leq \sup(\l A)$, and so $\l \sup(A) = \sup(\l A)$. The statement for infima can be proven similarly, or by using  $\inf A = - \sup(-A)$.
\end{proof}

\subsection{Supports and idempotents}

We define the set of idempotents in $\L$ by $\idempotents \defeq \{ \pi \in \L \colon \pi^2 = \pi\}$, which is a complete Boolean algebra with complementation $\pi^c \defeq 1-\pi$. By representation theory, these correspond to both the indicator functions of clopen subsets of $K$ (so that $\idempotents$ is isomorphic to the Boolean algebra of clopen subsets of $K$) as well as the \emph{components} of $1 \in \L$ (i.e., elements $\l \in \L^+$ such that $\l$ is disjoint from $(1-\l)$). Note that if $A \subseteq \idempotents$, then the least upper bound of $A$ in $\L$ is the indicator function of the closure of the union of the clopen subsets corresponding to the elements of $A$, and so it is in $\idempotents$; hence $\sup(A)$ (and also $\inf(A)$) is unambigiously defined. For $\l \in C_\infty(K)$, we denote by $\{\l \not= 0\}$ the set $\{\omega \in K \colon \l(\omega) \not= 0\}$; slight variations of this notation are defined similarly. In the sense of continuous functions, the support of a function $\l \in \L \subseteq C_\infty(K)$ is the clopen set $\ol{\{\l \not= 0\}}$; we will therefore define the \emph{support} of $\l \in \L$ by 
\begin{equation}\label{e:def of support in L}
\pi_\l \defeq \mathbf{1}_{\ol{\{\l \not= 0\}}} = \inf\{ \pi \in \idempotents \colon \pi \l = \l\} = \min\{ \pi \in \idempotents \colon \pi \l = \l\}.
\end{equation}

From vector lattice theory the band projection $P_\l$ onto the band generated by $\l$ satisfies $P_\l(\mu) = \pi_\l \mu$, and so $\pi_\l$ can also be defined as $P_\l(1)$. Note that $\l$ and $\mu$ are disjoint if and only if $\pi_\l \pi_\mu = 0$ (i.e., if their supports are disjoint).

\begin{lemma}\label{l:approximate_with_invertibles}
Let $\l \in \L$. Then there exists a sequence $(\l_n)$ of invertible elements with $| \l - \l_n| = n^{-1}$. Furthermore, if $\l \geq 0$, then $\l_n$ can be chosen to be decreasing, and if $\l \leq 0$, then $\l_n$ can be chosen to be increasing.
\end{lemma}

\begin{proof}
Consider a real $\l$ in $\L$. Now we choose 
\[\l_n \defeq \left(\l^+ + \frac{\pi_{\l^+}}n\right) - \left(\l^- + \frac{\pi_{\l^-}}{n}\right) \pm \frac{1-\pi_{\l^+}-\pi_{\l^-}}n.\] 
The conclusion follows from representation or vector lattice theory. If $\l$ is complex, it suffices to approximate its real part as above, since $\l_n$ is invertible if $\Re(\l_n)$ is invertible by \Cref{l:invertible}.
\end{proof}

The following lemma is inspired by \cite[Section~2.1]{edeko2023decomposition}.

\begin{lemma}\label{l:range_projection}
For every $\l \in \L^+$, we have
\[ 
\frac{\l}{\l + n^{-1}}\ \big\uparrow\ \pi_\l .
\]
\end{lemma}

\begin{proof}
Since $n^{-1} \leq \l + n^{-1}$ it follows by \Cref{l:invertible} that $\left( \l + n^{-1} \right)^{-1}$ exists in $\L^+$. Define $\l_n \defeq \frac{\l}{\l + n^{-1}}$. Using representation theory, for $\omega \in \{0 < \lambda < \infty \}$, it is clear that $\l_n(\omega) \uparrow 1$, and for $\omega \in \overline{\{0 < \lambda\}}^c$, we have $\lambda(\omega) = 0$ so $\l_n(\omega) = 0$ for each $n$. Thus $\mu$ is an upper bound of $(\l_n)$ if and only if $\mu \geq 1$ on $\overline{ \{ 0 < \lambda < \infty\} } = \overline{ \{0 < \lambda\} }$ and $\mu \geq 0$ on $\overline{\{0 < \lambda\}}^c$. Hence $\sup_n \l_n = \mathbf{1}_{\overline{\{0 < \lambda\}}} = \pi_\l$.  
\end{proof}

\begin{remark} For a proof of \Cref{l:range_projection} without using representation theory, let $B_\l$ be the band generated by $\l$. Since $\l_n \leq \l / n^{-1} = n \l$, $\l_n \in B_\l$, and so on $B^d_\l$ the sequence $\l_n$ is zero, and on $B_\l$, $\l$ is a weak unit. We will show that the supremum of $\l_n$ is the component of $1$ in $B_\l$. By the above discussion, it suffices to consider the case where $\l$ is a weak unit, and we have to show that the supremum of $\l_n$ is $1$, or equivalently, that the infimum of $1-\l_n$ is $0$. For this, note that
    $\dis{1 - \l_n = \frac{n^{-1}}{\l + n^{-1}} \geq 0}$. Let $a$ be another lower bound of $1-\l_n$. Then $a^+$ is also a lower bound of $1-\l_n$ and so $a^+(\l + n^{-1}) \leq n^{-1}$. Hence $a^+ \l \leq a^+ (\l + n^{-1}) \leq n^{-1}$, and so $n a^+ \l \leq 1$ for all $n \in \N$. Since $\L$ is Archimedean, $a^+ \l \leq 0$. By positivity of multiplication $a^+ \l \geq 0$ and so $a^+ \l = 0$. By \cite[Theorem~3.7(i)]{depagter}, $a^+$ and $\l$ are disjoint, and since $\l$ is a weak unit, $a^+ = 0$, showing that $a \leq 0$. Hence $0$ is the greatest lower bound of $1-\l_n$, as required.
\end{remark}

\begin{defn}\label{d:notation of taking a set power}
    Let $U \subseteq K$ be open and let $\mu$ be a function defined on a subset of $K$ such that $\mu|_U \in C(U)$. Extend $\mu|_U$ first by continuity to a function in $C_\infty(\ol{U})$ and then to $K$ by defining it to be zero on $\ol{U}^c$; this procedure defines an element we denote as $\mu^U \in C_\infty(K)$.
\end{defn}

Arguing as in the proof of \Cref{l:range_projection}, it follows that for $\l \in \L$,
\[ 
\pi_\l = \mathbf{1}^{\{0 < |\l| < \infty\}}.
\]
In this paper, the above definition will be used in cases where $\l \in \L$, $U = \{0 < |\l| < \infty\}$, and $\mu$ is some expression involving $\l$; in this case, note that $\l(\omega) = 0$ whenever $\omega \in \ol{U}^c$. An illustration is the following corollary of \Cref{l:range_projection}.

\begin{corol}\label{c:inf of weird fraction equals zero} 
Let $\l \in \L^+$. Then
\[ 
 \left( \frac{n^{-1}}{\l + n^{-1}} \right)^{\{0 < \l < \infty\}} \Big\downarrow\ 0.
\]
\end{corol}
\begin{proof}
For $\omega \in \{0 < \l < \infty\}$, we have
\[
\frac{n^{-1}}{\l(\omega) + n^{-1}} = 1  - \frac{\l(\omega)}{\l(\omega) + n^{-1}},
\]
and so by \Cref{l:range_projection} we obtain
\[
\left( \frac{n^{-1}}{\l+ n^{-1}} \right)^{\{0 < \l < \infty\}} = \mathbf{1}^{\{0 < |\l| < \infty\}} - \frac{\l}{\l + n^{-1}} = \pi_\l - \frac{\l}{\l + n^{-1}} \big\downarrow\ 0. \qedhere
\]
\end{proof}

\section{$\L$-normed spaces}\label{s:L-normed spaces}

\subsection{Basic definitions}

The next definition emphasizes the connection with classical linear algebra.
\begin{defn}
   If $X$ and $Y$ are $\L$-modules, then a map $T \colon X \to Y$ is \emph{$\L$-linear} if it is an $\L$-module homomorphism. The $\L$-module of $\L$-linear maps between $X$ and $Y$ is denoted by $L(X,Y)$. If $Z$ is another $\L$-module, then a map $\phi \colon X \times Y \to Z$ is called \emph{$\L$-bilinear} if it is $\L$-linear in each variable. An \emph{$\L$-algebra} is defined to be an $\L$-module $A$ equipped with an associative $\L$-bilinear map $(a,b) \mapsto ab$ from $A \times A$ to $A$.
\end{defn}
We will denote the zero element of an $\L$-module by $\0$.

\begin{defn}\label{defn: l-norm and l-seminorm}
Let $X$ be an $\L$-module. A map $\norm{\cdot} \colon X \to \L^+$ is called an \emph{$\L$-seminorm} on $X$ if the following two conditions are satisfied: For all $\l \in \L$ and $x,y \in X$, we have
    \begin{enumerate}
        \item[(i)] (Absolute homogeneity) $\norm{\l x} = |\l| \norm{x}$,
        \item[(ii)] (Triangle inequality) $\norm{x+y} \leq \norm{x} + \norm{y}$.
    \end{enumerate}
If additionally, we have 
    \begin{enumerate}
        \item[(iii)] (Positive definiteness) $\norm{x} = 0$ implies that $x=\0$,
    \end{enumerate}
then $\norm{\cdot} \colon X \to \L^+$ is called an \emph{$\L$-norm}. An \emph{$\L$-(semi)normed space} is defined to be an $\L$-module equipped with an $\L$-(semi)norm. An $\L$-algebra $A$ equipped with a norm is called an \emph{$\L$-normed algebra} if $\norm{ab} \leq \norm{a}\norm{b}$ for all $a,b \in A$.
\end{defn}

Note that $\norm{\0} = \norm{0 \cdot \0} = |0| \norm{\0} = 0$ in all $\L$-seminormed spaces. As usual, if $X$ is an $\L$-normed space we denote the \textit{unit ball} of $X$ as
\[
        B_X := \{ x\in X \colon \norm{x} \leq 1\}.
\]
By the properties of the modulus, $( \L, \left| \cdot \right| )$ is an $\L$-normed algebra.


\subsection{Convergence and completeness}

The notation $\Eps \searrow 0$ means that $\Eps$ is a subset of $\L^+$ with $\inf \Eps = 0$. In this notation $\Eps$ need not be a directed set.

\begin{defn}
Let $X$ be an $\L$-normed space. A net $(x_\alpha)_{\alpha \in I}$ is said to \textit{converge} to $x \in X$, denoted $x_\alpha \to x$, if there exists a set $\Eps \searrow 0$ such that for every $\eps \in \Eps$ there exists an $\alpha_0 \in I$ satisfying
\[
        \alpha \geq \alpha_0 \Rightarrow \norm{x_\alpha - x} \leq \eps.
\]
In the sequel we will almost always omit the index set and denote nets by $(x_\alpha)$. Where convenient, we may denote $x_\alpha \to x$ as $\lim_{\alpha} x_\alpha = x$. A set $A \subseteq X$ is \emph{closed} if $A\ni x_\alpha \to x$ implies that $x\in A$, and if $Y$ is another $\L$-normed space and $f \colon X \to Y$, then $f$ is \emph{continuous} if $x_\alpha \to x$ implies $f(x_\alpha) \to f(x)$.
Furthermore, a net $(x_\alpha)$ is said to be \textit{Cauchy} if there exists a set $\Eps \searrow 0$ such that for every $\eps \in \Eps$ there exists an $\alpha_0$ satisfying
\[
        \alpha, \beta \geq \alpha_0 \Rightarrow \norm{x_\alpha - x_\beta} \leq \eps.
\]The space $X$ is \emph{complete}, or an \emph{$\L$-Banach space}, if every Cauchy net in $X$ converges. An \emph{$\L$-Banach algebra} is a complete $\L$-normed algebra.
\end{defn}

By \Cref{l:order_cont_mult}, if $\Eps \searrow 0$ then $2\Eps \searrow 0$ and so by the triangle inequality it follows that convergent nets are Cauchy. 

\begin{remark}
Given $\Eps \subseteq \L$, the set of infima of finite subsets of $\Eps$ is a directed set with the same lower bounds as $\Eps$. From this it follows that in the definitions of convergent and Cauchy nets we may replace $\Eps$ with a downwards directed set with infimum $0$, or with a net decreasing to $0$.  Therefore, in the case of the $\L$-normed space $\L$, convergent and Cauchy nets are precisely the \emph{order convergent} and \emph{order Cauchy} nets, respectively, see for instance \cite[Section 9]{Netconvergence}. Our definition more closely resembles the $\eps-N$ definition of convergence in a normed space. 
The key difference is that in the case of an $\L$-normed space, the set $\Eps$ depends on the particular net under consideration. 
 
We note that, in general, convergence in an $\L$-normed space $X$ is not topological; that is, there exists no topology on $X$ so that the convergent nets in $X$ are precisely those that convergence with respect to the topology. Indeed, (order) convergence in $\L$ is topological if and only if $\L$ is finite dimensional, see \cite{Orderconvergencenottopological}.  However, convergence in $\L$ defines a convergence structure in the sense of \cite{BeattieButzmann2002,Netconvergence}.  It can also be shown that convergence in any $\L$-normed space defines a convergence structure.
\end{remark}

Every subset $A$ of an $\L$-normed space $X$ is contained in a closed subset of $X$, namely in $X$.  It follows immediately from the definition of a closed set that the intersection of any collection of closed sets in closed.  Therefore we may define the closure of $A$ in the usual way.

\begin{defn}
    Let $X$ be an $\L$-normed space and $A\subseteq X$.  The \emph{closure} of $A$ is the smallest closed set that contains $A$, and is denoted $\overline{A}$. An \emph{adherent} point of $A$ is a limit of a converging net in $A$.
\end{defn}

\begin{remark}
    Let $X$ be an $\L$-normed space and $A\subseteq X$.  Since convergence in an $\L$-normed space is typically not topological, the set 
    \[
    \{ x \in X \colon \text{there is a net } (x_\alpha) \text{ in } A \text{ such that } x_\alpha \to x \}
    \]
    is in general not closed, and therefore not equal to the closure of $A$.
\end{remark}
It turns out that in certain cases (see \Cref{c:closure_equals_adherence}) the closure does equal the set of adherent points.

As in the classical case, we have the reverse triangle inequality as stated in the next lemma. The proof is identical to the classical case. 
\begin{lemma}\label{l:reverse_triangle}
Let $X$ be an $\L$-normed space and $x,y \in X$. Then 
\[
\big| \norm{x} - \norm{y} \big| \leq \norm{x -y}.
\]
\end{lemma}

\begin{corol}\label{c:continuity_of_norm}
    Let $X$ be an $\L$-normed space. Then $\norm{\cdot} \colon X \to \L$ is continuous.
\end{corol}

In the next theorem we prove the order continuity of addition. This follows from well-known facts about order convergence in $\L$, but we provide a proof from first principles for completeness.

\begin{theorem}\label{t:addition_continuous}
   Let $X$ be an $\L$-normed space, let $x_\alpha \to x$ and $y_\alpha \to y$. Then $x_\alpha + y_\alpha \to x+y$, so order convergence is jointly continuous. 
\end{theorem}
\begin{proof}
    Suppose $x_\alpha \to x$ and $y_\alpha \to y$. Then, there exist sets $\Eps_1, \Eps_2 \searrow 0$ such that for any $\eps_1 \in \Eps_1$ and for any $\eps_2 \in \Eps_2$ there exist $\alpha_1$ and $\alpha_2$ such that $\norm{x_\alpha - x} \leq \eps_1$ for all $\alpha \geq \alpha_1$ and $\norm{y_\alpha - y} \leq \eps_2$ for all $\alpha \geq \alpha_2$. Define $\Eps \coloneqq \Eps_1 + \Eps_2$, then $\Eps \searrow 0$ by \Cref{l:sup_additive}. For any $\eps = \eps_1 + \eps_2 \in \Eps$, let $\alpha_1$ and $\alpha_2$ be as above. By directedness there exists some $\alpha_0 \geq \alpha_1, \alpha_2$. Then for all $\alpha \geq \alpha_0$,
     \[   \norm{(x_\alpha + y_\alpha) - (x + y)} \leq \norm{x_\alpha - x} + \norm{y_\alpha - y} \leq \eps_1 + \eps_2 = \eps.  \qedhere \]
\end{proof}

\begin{prop}\label{p: Limits in L-normed spaces are unique}
Limits in $\L$-normed spaces are unique.      
\end{prop}

\begin{proof}
Let $X$ be an $\L$-normed space with $x,y\in X$ and suppose $x_\alpha \to x$ and $x_\alpha \to y$. By \Cref{t:addition_continuous} and \Cref{c:continuity_of_norm}, $\norm{x - x_\alpha}+ \norm{x_\alpha - y} \to 0$, and so there exists $\Eps\searrow 0$ such that for every $\eps\in \Eps$ there is an $\alpha_0$ such that
\[
        \norm{x - x_\alpha}+ \norm{x_\alpha - y} \leq \eps
\]for all $\alpha\geq \alpha_0$. For $\alpha \geq \alpha_0$, we have 
\[
        0 \leq \norm{x - y} \leq \norm{x_\alpha - x} + \norm{x_\alpha - y} \leq \eps
\] which implies that $\norm{x - y} \leq \inf\Eps = 0$, hence $x = y$. 
\end{proof}

\begin{prop}\label{p:conv_net_bounded_tail}
Let $(x_\alpha)$ be a Cauchy net in an $\L$-normed space $X$, then $(x_\alpha)$ has a bounded tail. In particular, if $x_\alpha \to x$, then $(x_\alpha)$ has a bounded tail.
\end{prop}
\begin{proof}
There exists a set $\Eps \searrow 0$ such that for all $\eps \in \Eps$ there exists an $\alpha_0$ such that $\norm{x_\alpha - x_\beta} \leq \eps$ for all $\alpha, \beta \geq \alpha_0$. 
Fix an $\eps \in \Eps$ with its corresponding $\alpha_0$, then $\norm{x_\alpha} \leq \norm{x_\alpha - x_{\alpha_0}} + \norm{x_{\alpha_0}} \leq \eps + \norm{x_{\alpha_0}}$ for all $\alpha \geq \alpha_0$. 
\end{proof}

An $\L$-bilinear map $\phi \colon X \times Y \to Z$ is called \emph{bounded} if there exists an $M \in \L^+$ with $\norm{\phi(x,y)}_Z \leq M \norm{x}_X \norm{y}_Y$ for all $x \in X$ and $y \in Y$. If $M\leq 1$, we say that $\phi \colon X \times Y \to Z$ is \textit{contractive}. 

\begin{theorem}\label{t:bilinear_cont}
Let $X$, $Y$, and $Z$ be $\L$-normed spaces and let $\phi \colon X \times Y \to Z$ be a bounded $\L$-bilinear map. If $x_\alpha \to x$ and $y_\alpha \to y$, then $\phi(x_\alpha, y_\alpha) \to \phi(x,y)$, so that $\phi$ is jointly continuous.
\end{theorem}

\begin{proof}
Since $\phi$ is bounded, there exists an $M \in \L^+$ such that $\norm{\phi(x,y)}_Z \leq M \norm{x}_X \norm{y}_Y$ for all $x \in X$ and $y \in Y$. Suppose $x_\alpha \to x$ and $y_\alpha \to y$. Then there exist $\Eps_1, \Eps_2 \searrow 0$ such that for all $\eps_1 \in \Eps_1$ and for all $\eps_2 \in \Eps_2$ there exist $\alpha_1$ and $\alpha_2$ satisfying $\norm{x_\alpha - x}_X \leq \eps_1$ for all $\alpha \geq \alpha_1$ and $\norm{y_\alpha - y}_Y \leq \eps_2$ for all $\alpha \geq\alpha_2$.

By Proposition \ref{p:conv_net_bounded_tail}, there exists a $\lambda \in \L^+$ and an $\alpha_3$ such that $\norm{y_\alpha}_Y \leq \lambda$ for all $\alpha \geq \alpha_3$. By Lemmas \ref{l:sup_additive} and \ref{l:order_cont_mult}, $M \lambda \Eps_1 + M \norm{x}_X \Eps_2 \searrow 0$. Now, for any $\eps_1 \in \Eps_1$ and $\eps_2 \in \Eps_2$, choose $\alpha_1$ and $\alpha_2$ as above, and then let $\alpha_0 \geq \alpha_1, \alpha_2, \alpha_3$. Then, for $\alpha \geq \alpha_0$,
\begin{align*}
\norm{\phi(x_\alpha, y_\alpha) - \phi(x, y)}_Z
& \leq \norm{\phi(x_\alpha, y_\alpha) - \phi(x, y_\alpha)}_Z + \norm{\phi(x, y_\alpha) - \phi(x, y)}_Z \\
& = \norm{\phi(x_\alpha - x, y_\alpha)}_Z + \norm{\phi(x, y_\alpha - y)}_Z \\
& \leq M \norm{x_\alpha - x}_X \norm{y_\alpha}_Y + M \norm{x}_X \norm{y_\alpha - y}_Y \\
& \leq M \lambda \eps_1 + M \norm{x}_X \eps_2. \qedhere
\end{align*}
\end{proof}

Since scalar multiplication is a contractive $\L$-bilinear map, the previous result yields the following corollary.

\begin{corol}\label{c:Scalar mult is continuous}
Let $X$ be an $\L$-normed space. Then scalar multiplication $(\l,x) \mapsto \l x$ from $\L \times X$ to $X$ is jointly continuous.
\end{corol}

If $\L^+ \ni \l_\alpha \to \l$, then $\l_\alpha = |\l_\alpha| \to |\l|$ and so $\l = |\l|$ which implies that $\L^+$ is closed. It now follows by the continuity of addition and scalar multiplication that sets of the form $\{ \lambda \in \L : \lambda \leq \mu \}$ and  $\{ \lambda \in \L : \lambda \geq \mu \}$ for $\mu \in \L_{\R}$ are closed. 



We now show that $\L$ is indeed an $\L$-Banach space. The proof is found in \cite[Proposition 9.10]{Netconvergence} and we reproduce it here for the convenience of the reader.

\begin{prop}\label{p: L is complete}
$\L$ is a complete $\L$-normed space. 
\end{prop}

\begin{proof}
Consider a Cauchy net $(\lambda_\alpha)$ in $\L$. By \Cref{p:conv_net_bounded_tail}, passing to a tail of the net if necessary, we may assume that for some $u\in \L^+$, $-u\leq\lambda_\alpha\leq u$ hold for all $\alpha$. By Dedekind completeness of $\L$, for every $\alpha$, the elements
\[
        \mu_\alpha \defeq \sup_{\beta \geq \alpha} \lambda_\beta \quad \text{and} \quad \nu_\alpha \defeq \inf_{\beta \geq \alpha} \lambda_\beta.
\]
are defined. Note that $(\mu_\alpha)$ is decreasing, $(\nu_\alpha)$ is increasing, and $-u\leq \nu_\alpha \leq \lambda_\alpha \leq \mu_\alpha\leq u$ for all $\alpha$.  Therefore
\[
        \mu \defeq \inf_\alpha  \mu_\alpha, \qquad \nu \defeq \sup_\alpha \nu_\alpha
\] 
exist in $\L$. We claim that $\mu=\nu$ and $\lambda_\alpha\to \mu$.  

By definition of a Cauchy net there exists $\Eps\searrow 0$ so that for every $\eps\in\Eps$ there exists $\alpha_\eps\in I$ so that $|\l_\alpha-\l_\beta|\leq \eps$ for all $\alpha,\beta\geq \alpha_\eps$.  Fix $\eps \in \Eps$.   Then, for all $\alpha,\beta \geq \alpha_\eps$,
\[
-\eps \leq \lambda_\alpha - \lambda_\beta \leq \eps.
\]
Therefore
\[
\mu_{\alpha_\eps} - \nu_{\alpha_\eps} = \sup_{\alpha, \beta \geq \alpha_\eps} (\lambda_\alpha - \lambda_\beta) \leq \eps.
\]
It now follows that
\[
\mu-\nu = \inf_{\alpha, \beta} (\mu_\alpha - \nu_\beta) \leq \inf_{\eps \in \Eps}(\mu_{\alpha_\eps} - \nu_{\alpha_\eps})  \leq \inf \Eps = 0
\]
so that $\mu \leq \nu$.  In the same way, $\nu \leq \mu$ so that $\mu = \nu$ as claimed.

It remains to show that $\lambda_\alpha \to \mu$.  To see that this is so, let $\eps\in\Eps$ and $\alpha,\beta \geq \alpha_\eps$.  We observe that
\[
\lambda_\alpha - \mu \leq \lambda_\alpha - \nu_\beta \leq \mu_\alpha - \nu_\beta \leq \mu_{\alpha_\eps} - \nu_{\beta_\eps} \leq \eps
\]
and
\[
\lambda_\alpha - \mu \geq \nu_\alpha -\mu \geq \nu_\alpha - \mu_\beta \geq \nu_{\alpha_\eps} - \mu_{\beta_\eps} \geq -\eps .
\]
Therefore
\[
|\lambda_\alpha - \mu| = (\lambda_\alpha - \mu) \vee (\mu-\lambda_\alpha) \leq \eps.
\]
Hence $\lambda_\alpha \to \mu$ as claimed.
\end{proof}

\begin{remark}
In the vector lattice literature, an Archimedean vector lattice is often called \emph{order complete} if every order Cauchy net is order convergent.  It is shown in \cite[Proposition 9.10]{Netconvergence} that an Archimedean vector lattice is order complete if and only if it is Dedekind complete. Proposition \ref{p: L is complete} is a special case of this result.

It is well known that for a compact Hausdorff space $K$, $C(K)$ is Dedekind complete if and only if $K$ is Stonean, see for instance \cite[Proposition 2.1.4]{Meyer-Nieberg1991}.  Therefore $C(K)$ is order complete if and only if it is Dedekind complete, if and only if $K$ is Stonean.  This equivalence is also established in \cite[Proposition~1.9]{edeko2023decomposition}.
\end{remark}

A \emph{step function} $\l \in \L$ is an $\C$-linear combination of disjoint idempotents.

\begin{theorem}[Freudenthal Spectral Theorem]\label{t:freudenthal} 

$\phantom{xx}$

\begin{enumerate}[(i)] 
\item Let $\l\in \L^+$. Then there exists a sequence $(\l_n)$ of step functions such that $\l_n \uparrow \l$. 
\item For any $\l \in \L$, there is a sequence $(\l_n)$ of step functions such that $\l_n \to \l$.
\end{enumerate}
\end{theorem}
The proof of (i) can be found in \cite[Theorem 33.3]{Zaanen}, under a weaker assumption of vector lattices with the principal projection property. Note that (ii) follows from (i) and the fact that any $\l \in \L$ can be written as $\l=\l_1-\l_2+i\l_3-i\l_4$, with $\l_i\in \L^+$ for all $i\in\{1,2,3,4\}.$

\subsection{Supports, separatedness, and normalisation}

Inspired by \eqref{e:def of support in L} we define, for $x$ in some $\L$-module X, the support $\pi_x$ of $x$ by
\begin{equation}\label{e:def of support in X}
    \pi_x \defeq \inf\{ \pi \in \idempotents \colon \pi x = x\}.
\end{equation}
If $X$ is an $\L$-normed space and $x\in X$, we may view $P_x \defeq \{ \pi \in \idempotents \colon \pi x = x\}$ as a decreasing net indexed by itself with reverse ordering, and then $P_x \to \pi_x$ in $\L$. It then follows by continuity of scalar multiplication that $\pi_x x = x$ and so $\pi_x = \min\{ \pi \in \idempotents \colon \pi x = x\}$ (if an $\L$-module is not normed, this equality might not hold).

\begin{lemma}\label{l:support_equals_support_of_norm}
Let $X$ be an $\L$-normed space. Then $\pi_x = \pi_{\norm{x}}$.
\end{lemma}
\begin{proof}
Let $\pi \in \P$ and $x \in X$. If $\pi x = x$, then $\pi \norm{x} = \norm{\pi x} = \norm{x}$. Conversely, if $\pi \norm{x} =  \norm{x}$, then $\norm{x - \pi x} = (1 - \pi) \norm{x} = \norm{x} - \pi \norm{x} = 0$ and so $x - \pi x = \bf{0}$, thus $\pi x = x$. The lemma now follows from the definition of the support.
\end{proof}

\begin{defn}\label{d:disjointness in vector spaces}
Let $X$ be an $\L$-module. We define $x,y \in X$ to be \emph{separated} if there exists a $\pi \in \idempotents$ with $\pi x = x$ and $\pi^c y = y$. 
\end{defn}
It is clear that if $\pi_x x = x$ and $\pi_y y = y$ (which always holds in $\L$-normed spaces), then $x$ and $y$ are separated if and only if $\pi_x \pi_y = 0$ (i.e., $x$ and $y$ have disjoint supports). Note that the $\L$-module definition of separatedness extends the notion of disjointness in $\L$.

A map $\sigma \colon X \to Y$ between $\L$-vector spaces is called \emph{$\P$-homogeneous} if $\sigma(\pi x) = \pi \sigma(x)$ for all $\pi \in \idempotents$ and $x \in X$. Important examples are $\L$-linear maps, seminorms, and sublinear maps (see \Cref{d:sublinear}).

\begin{lemma}\label{l:disjoint-implies-additive}
    Let $X$ and $Y$ be $\L$-modules and let $\sigma \colon X\to Y$ be $\P$-homogeneous. 
    \begin{enumerate}
        \item  $\pi_{\sigma(x)} \leq \pi_x$ (so $\sigma$ reduces supports).
        \item If $x$ and $y$ are separated, then $\sigma(x+y)=\sigma(x)+\sigma(y)$.
    \end{enumerate}
\end{lemma}
\begin{proof}
(i): Let $\pi \in \idempotents$ be such that $\pi x = x$, then $\pi \sigma(x) = \sigma(\pi x) = \sigma(x)$. Hence
\[
\{\pi \in \idempotents \colon \pi x = x\} \subseteq \{\pi \in \idempotents \colon \pi \sigma(x) = \sigma(x)\},
\]
and taking the infimum over $\pi \in \idempotents$ yields (i).

(ii): Let $\pi \in \idempotents$ be such that $\pi x =x$ and $\pi^c y = y$. Then $\pi^c x = 0$ and $\pi y = 0$, and so
\begin{align*}
        \sigma(x + y) &= \pi \sigma(x + y) + \pi^c \sigma(x + y)\\
                     &= \sigma( \pi x + \pi y ) + \sigma(\pi^c x + \pi^c y ) \\
                     &= \sigma(x) + \sigma(y). \qedhere
\end{align*}
\end{proof}

Inspired by \cite[Section~2.1]{edeko2023decomposition}, for an $\L$-normed space $X$, we call an element $x\in X$ \textit{normalised} if $\norm{x} \in \idempotents$, and if $X$ is an $\L$-Banach space, we define the \textit{normalisation} of $x\in X$ as 
\begin{align}\label{e: Defn of normalisation}
        n_x \defeq \lim_{n \to \infty} \frac{x}{\norm{x}+ n^{-1}}.
\end{align}
The existence of the limit in \eqref{e: Defn of normalisation} is demonstrated in the following lemma. 

\begin{lemma}\label{l: Normalisation exists in an L-Banach space}
Let $X$ be any $\L$-Banach space. For every $x\in X$, the normalisation $n_x$ exists.
\end{lemma}
\begin{proof}
Fix $x\in X$ and for every $n\in \N$ define $\lambda_n:=\frac{1}{\norm{x}+n^{-1}}$. By \Cref{l:range_projection}, the sequence $\left(\l_n\norm{x}\right)$ is convergent, hence Cauchy in $\L$. Since, for every $n,m\in \N$, we have
\[
\norm{\lambda_n x - \lambda_m x} = \left| \l_n - \l_m \right| \norm{x} = \big| \l_n \norm{x} - \l_m \norm{x} \big|,
\]
the sequence $\left(\l_n x\right)$ is also Cauchy, and therefore is convergent in $X$. 
\end{proof}

We record some basic properties of supports and normalisations in the following proposition, which largely follows \cite[Lemma~2.3]{edeko2023decomposition} but we have added more details. 

\begin{prop}\label{p: normalise}
Let $X$ be an $\L$-Banach space. For $x\in X$, we have
\begin{enumerate}
    \item $\dis{\norm{n_x} = \sup_{n\in \N} \frac{\norm{x}}{\norm{x}+n^{-1}}} = \pi_{\norm{x}} = \pi_x$; in particular, $n_x$ is normalised.
    \item $\norm{x} n_x = x$.
    \item If $\norm{x}$ is invertible, then $\dis{n_x = \frac{x}{\norm{x}}}$.   
    \item $x$ is normalised $\iff \norm{x}x = x \iff x = n_x \iff \pi_x = \norm{x}$.  
\end{enumerate}
\end{prop}

\begin{proof}
The statement in (i) follows from the continuity of the norm, \Cref{l:range_projection}, and \Cref{l:support_equals_support_of_norm}. 
The statement in (ii) follows from (i):
\[
    \norm{x} n_x  = \lim_{n\to \infty} \frac{\norm{x}}{\norm{x}+ n^{-1}}x = \pi_x x = x.
\]
The statement in (iii) follows immediately from (ii). We prove the equivalences stated in (iv). First, assume that $x\in X$ is normalised, i.e. $\norm{x} \in \idempotents$. Then by two applications of (ii), we have
\[
        \norm{x}x = \norm{x}\left( \norm{x}n_x\right) = \norm{x}^2 n_x = \norm{x}n_x = x.
\]Next, assume that $\norm{x}x = x$. Using (ii), this gives us
\[
        n_x = \lim_{n \to \infty} \frac{x}{\norm{x}+n^{-1}} =  \lim_{n \to \infty} \frac{\norm{x} x }{\norm{x}+n^{-1}} = \norm{x} \lim_{n \to \infty} \frac{x}{\norm{x}+n^{-1}} = \norm{x} n_x = x.  
\]
If $x = n_x$, then $\pi_x = \norm{n_x} = \norm{x}$ by (i). Lastly, if $\pi_x = \norm{x}$, then $\norm{x} = \pi_x \in \P$, i.e., $x$ is normalised.
\end{proof}

For the $\L$-Banach space $\L$, we can improve the result in \Cref{p: normalise}~(iii) by giving a pointwise characterisation of the normalisation of an element $\l \in \L$. Recall the definition of $\mu^U$ from \Cref{d:notation of taking a set power}.



\begin{lemma}\label{l: Pointwise characterisation of normalisation in L}
Let $\l \in \L$. Then 
\[
        n_\l = \left(\frac{\l}{|\l|} \right)^{\{ 0 < |\l| < \infty \}}.
\]
\end{lemma}

\begin{proof}
Denote $U \defeq \{ 0 < |\l| < \infty \}$ and define $\mu \defeq \left( \l/|\l|\right)^U$. To show that $n_\l = \mu$, it suffices to show that the sequence $s_n \defeq \l/\left( |\l| + n^{-1} \right)$ converges to $\mu$. For $n\in \N$, define $\mu_n \defeq \left( n^{-1}/ ( |\l| + n^{-1} ) \right)^U$. Since $\left| \mu - s_n \right| = \mu_n$ for every $n\in \N$ and $\mu_n \downarrow 0$ in $\L^+$ (\Cref{c:inf of weird fraction equals zero}), we conclude that $s_n \to \mu$.
\end{proof}


        

From \Cref{p: normalise}~(ii), for any $\l\in \L$, we have $|\l| n_{\l} = \l$. The pointwise characterisation of normalisations in \Cref{l: Pointwise characterisation of normalisation in L} gives the following complementary result. 

\begin{corol}\label{c: Rotation by scalar multiplication}
For any $\l \in \L$, we have $\l n_{\overline{\l}} = |\l|$.   
\end{corol}

\subsection{Realizing spacial support for $\L$-Banach spaces}

In this section we define what we mean by the support of an $\L$-vector space. Furthermore, we show that, in an $\L$-Banach space, its support is always realized by a normalized element in the space.

\begin{defn}
    Let $X$ be an $\L$-vector space. 
    We define the \emph{support of $X$} as  $\pi_{X}:=\sup\{ \pi_{x}\in \idempotents : x\in X\}.$ 
\end{defn}

The next lemma and proof are \cite[Lemma~2.4]{edeko2023decomposition}, where we added some more details to the proof.
\begin{lemma}
    \label{l:L-ban-support-is-realized}Let $X$ be an $\L$-Banach
    space. There exists a normalised element $x\in X$ so that $\norm x = \pi_{x}=\pi_{X}.$
\end{lemma}

\begin{proof}
    Define the set of normalised elements as $\mathcal{N}:=\{x\in X \colon \norm x \in \idempotents\}$.
    For $x,y\in\mathcal{N}$, we define $x\preceq y$ to mean $x=\pi_{x}y$.
    The relation $\preceq$ is a partial order. We
    aim to apply Zorn's Lemma to obtain a maximal element in $\mathcal{N}$.
    Let $\mathcal{K}$ be an arbitrary chain in $\mathcal{N}$ and define
    $\pi:=\sup\{\norm x \colon x\in\mathcal{K}\}$. Since $\idempotents$ is
    a complete Boolean algebra, we have that $\pi \in\idempotents$.
    The chain $\mathcal{K}$ is a directed set (with respect to $\preceq$)
    and we view $\mathcal{K} \subseteq X$ as a net indexed by itself.

    We claim that $\mathcal{K}$ is a Cauchy net and that the limit of
    this net is an upper bound (with respect to $\preceq$) for $\mathcal{K}$.
    Notice, for $a,x,y\in\mathcal{K}$ with $a\preceq x,y$, we have $a=\pi_{a}x$
    and $a=\pi_{a}y$, and that $\norm x,\norm y\leq \pi$. Therefore
    \begin{align*}
        \norm{x-y} & \leq\norm{x-a}+\norm{a-y}          \\
                   & =\norm{x-\pi_{a}x}+\norm{\pi_{a}y-y}   \\
                   & =(1-\pi_{a})\norm x+(1-\pi_{a})\norm y \\
                   & =(1-\pi_{a})(\norm x+\norm y)        \\
                   & \leq(1-\pi_{a})2\pi=2(\pi_{a}^c \pi)
    \end{align*}
    With 
    $\mathcal{L}:=\{a\in\mathcal{K}:\pi_{a}=\norm a\leq \pi\}$ ordered
    by $\preceq$, the net 
    $(\pi_{a}^{c}\pi)_{a\in\mathcal{L}}     \subseteq       \idempotents$,
    by definition of $\pi,$ decreases to zero in $\L$. 
    Therefore $\mathcal{K}\subseteq X$
    is a Cauchy net and, since $X$ is an $\L$-Banach space, converges
    to some $x_{0}\in X$. From
    $\norm{x_{0}}
        =\norm{\lim_{x\in\mathcal{K}}x}
        =\lim_{x\in\mathcal{K}}\norm x
        =\sup_{x\in \mathcal{K}}\pi_{x}
        =\pi$
    we have that $x_{0}\in\mathcal{N}$. 
    Furthermore, for $a\in\mathcal{K}$   we have     
    $\pi_{a}x_{0}        
    =\pi_{a}\lim_{x\in\mathcal{K}}x
    =\lim_{x\in\mathcal{K}}\pi_{a}x
    =\lim_{x\in\mathcal{K},a\preceq x}\pi_{a}x=a$,
    from which we see that $x_{0}$ is an upper bound for the chain $\mathcal{K}.$

    By Zorn's Lemma, $\mathcal{N}$ contains a maximal element (with respect
    to $\preceq$) which we denote $x\in\mathcal{N}$. Let $y\in X$ be
    arbitrary and define $z:=x+\pi_{x}^{c}n_{y}.$ Then, since
    \begin{align*}
        \norm z 
        &=\norm{x+\pi_{x}^{c}n_{y}} \\
        &=\norm{\pi_{x}x+\pi_{x}^{c}n_{y}}  \\
        & = \norm{\pi_{x}x} + \norm{\pi_{x}^{c}n_{y}} \\
        &=\pi_{x}\norm{x} + \pi_{x}^{c}\norm{n_{y}} \\
        &=\pi_{x}+\pi_{x}^{c}\norm{n_{y}}\in\idempotents,
    \end{align*}
    we have that $z\in\mathcal{N}$. 
    Also 
    $\pi_{x}z
    =\pi_{x}(x+\pi_{x}^{c}n_{y})
    =\pi_{x}x+\pi_{x}\pi_{x}^{c}n_{y}
    =\pi_{x}x
    = x$,
    so that $x\preceq z$, and by $\preceq$-maximality of $x$, we have
    $z=x$. This implies, for all $y\in X$, that 
    $0 = \norm{\pi_{x}^{c}n_{y}} = \pi_{x}^{c}\pi_{y}$,
    i.e., for all $y\in X$, $\pi_{y}\leq \pi_{x}.$ Therefore 
    $\pi_{x}=\sup\{\pi_{y}:y\in X\}=\pi_{X}.$
\end{proof}

\subsection{Spaces of bounded functions}

We record some examples of $\L$-Banach spaces which come from the classical theory. 

\begin{example}
Let $S$ be a non-empty set and let $Y$ be an $\L$-normed space (in particular, the reader should keep the important example of $Y = \L$ in mind). Denote by $\ell^\infty(S, Y)$ the set of functions $f \colon S \to Y$ such that there exists an $M \in \L^+$ with $\norm{f(s)}_Y \leq M$ for all $s \in S$. For $f \in \ell^\infty(S,Y)$ we define $\norm{f}_\infty \defeq \sup_{s \in S} \norm{f(s)}_Y$; this is well-defined by the Dedekind completeness of $\L$. For $f \in \ell^\infty(S,Y)$ and $\l \in \L$, by \Cref{l:order_cont_mult} we obtain
\[
\norm{\l f}_\infty = \sup_{s \in S} \norm{\l f(s)}_Y = \sup_{s \in S} |\l| \norm{f(s)}_Y = |\l| \sup_{s \in S} \norm{f(s)}_Y = |\l| \norm{f}_\infty
\]
and for $f,g \in \ell^\infty(S,Y)$ we have
\[ 
\norm{f+g}_\infty = \sup_{s \in S} \norm{f(s) + g(s)}_Y \leq \sup_{s \in S} (\norm{f(s)}_Y + \norm{g(s)}_Y) \leq \norm{f}_\infty + \norm{g}_\infty.
\]
Since $\norm{f}_\infty = 0$ clearly implies $f = \0$, it follows that $\ell^\infty(S,Y)$ is an $\L$-normed space.    
\end{example}

\begin{theorem}\label{t:completeness_bounded_functions}
    Let $S$ be a non-empty set and let $Y$ be an $\L$-Banach space. Then $\ell^\infty(S,Y)$ is an $\L$-Banach space.
\end{theorem}

\begin{proof}   
Let $(f_\alpha)$ be a Cauchy net in $\ell^\infty(S, Y)$. Then there exists an $\Eps \searrow 0$ such that for all $\eps \in \Eps$ there exists an $\alpha_0$ satisfying
\[ 
        \alpha, \beta \geq \alpha_0 \Rightarrow \norm{f_\alpha - f_\beta}_\infty \leq \eps.
\]
So, for any $s \in S$ and for $\alpha, \beta \geq \alpha_0$ we have
\begin{equation}\label{e:est1}
        \norm{ f_\alpha(s) - f_\beta(s)}_Y \leq \norm{f_\alpha - f_\beta}_\infty \leq \eps; 
\end{equation}
therefore, $(f_\alpha(s))$ is Cauchy in $Y$. Since $Y$ is complete, $f_\alpha(s) \to f(s)$ for some $f \colon S \to Y$. Fix $\eps \in \Eps$ and $s\in S$. Since $f_\beta(s) \to f(s)$, there exists a set $\Eps_s \searrow 0$ such that, for all $\eta \in \Eps_s$, there is a $\beta_s$ such that $\norm{f_\beta(s) - f(s)} \leq \eta$ for all $\beta \geq \beta_s$. Now for $\eta \in \Eps_s$, choose $\beta \geq \beta_s, \alpha_0$, then for $\alpha \geq \alpha_0$,
\[
        \norm{f_\alpha(s) - f(s)}_Y \leq \norm{f_\alpha(s) - f_\beta(s)}_Y + \norm{f_\beta(s) - f(s)}_Y \leq \eps + \eta.
\] 
Hence $\norm{f_\alpha(s) - f(s)}_Y$ is a lower bound of $\eps + \Eps_s$ and so (still for $\alpha \geq \alpha_0$)
\begin{equation}\label{e:est2}
        \norm{f_\alpha(s) - f(s)}_Y \leq \inf( \eps + \Eps_s) = \eps. 
\end{equation}
Alternatively, \eqref{e:est2} can also be derived by taking the limit for $\beta$ in \eqref{e:est1}, using the continuity of the norm, and the fact that $\{\l \in \L \colon \l \leq \eps\}$ is closed.
    
Now $\norm{f(s)}_Y \leq \norm{f_{\alpha_0}(s) - f(s)}_Y + \norm{f_{\alpha_0}(s)}_Y$ shows that $f \in \ell^\infty(S,Y)$, and taking the supremum over $s \in S$ in \eqref{e:est2} yields that $\norm{f_\alpha - f}_\infty \leq \eps$ for all $\alpha \geq \alpha_0$ and so $f_\alpha \to f$ in $\ell^\infty(S,Y)$.
\end{proof}

Still considering a non-empty set $S$ and an $\L$-normed space $Y$, we define the subspace $c_0(S,Y)$ of $\ell^\infty(S,Y)$ as the set of functions $f \colon S \to Y$ for which there exists an $\Eps \searrow 0$ such that for all $\eps \in \Eps$, there is a cofinite subset $C \subseteq S$ with $\norm{f(s)}_Y \leq \eps$ for all $s \in C$.

\begin{prop}\label{p:c_0 closed}
Let $S$ be a non-empty set and $Y$ an $\L$-normed space. Then $c_0(S,Y)$ is a closed subspace of $\ell^\infty(S,Y)$. Moreover, if $Y$ is an $\L$-Banach space, then $c_0(S,Y)$ is an $\L$-Banach space.   
\end{prop}

\begin{proof}
Let $(f_\alpha)$ be a net in $c_0(S,Y)$ converging to $f \in \ell^\infty(S,Y)$. Since $f_\alpha \in c_0(S,Y)$, for each $\alpha$ there exists $\Eps_\alpha \searrow 0$ such that for all $\eps \in \Eps_\alpha$, there is a cofinite subset $C \subseteq S$ with $\norm{f_\alpha(s)}_Y \leq \eps$ for all $s \in C$. Since $f_\alpha \to f$, there exists $H \searrow 0$ such that for all $\eta \in H$, there exists $\alpha_\eta$ such that $\norm{f - f_\alpha}_\infty \leq \eta$ for all $\alpha \geq \alpha_\eta$. Define 
\[ 
\Eps \defeq \{ \eta + \eps \colon \eta \in H,\ \eps \in \Eps_{\alpha_\eta} \} = \bigcup_{\eta \in H} \left( \eta + \Eps_{\alpha_\eta} \right),
\]
then 
\[
\inf \Eps = \inf_{\eta \in H} \left[ \inf \left( \eta + \Eps_{\alpha_\eta} \right) \right] = \inf_{\eta \in H} \eta = 0.
\]
Now take $\eta \in H$ and $\eps \in \Eps_{\alpha_\eta}$, so that $\eta + \eps$ is an arbitrary element of $\Eps$. Let $C \subseteq S$ be a cofinite subset $C \subseteq S$ with $\norm{f_{\alpha_\eta}(s)}_Y \leq \eps$ for all $s \in C$. Then for all $s \in C$,
\[
\norm{f(s)}_Y \leq \norm{f(s) - f_{\alpha_\eta}(s)}_Y + \norm{f_{\alpha_\eta}(s)}_Y \leq \eta + \eps.
\]
Hence $f \in c_0(S,Y)$, as required. 

The final statement follows from the easily verifiable fact that a subspace of an $\L$-Banach space is closed if and only if it is complete.
\end{proof}

\subsection{Operators between $\L$-normed spaces}


\begin{defn}
Let $X$ and $Y$ be $\L$-normed spaces. A function $f \colon X \to Y$ is called \textit{uniformly continuous} if $x_\alpha - y_\alpha \to \0$ implies $f(x_\alpha) - f(y_\alpha) \to \0$. Furthermore, $f \colon X \to Y$ is called \textit{Lipschitz continuous} if there exists  $\lambda \in \L^+$ such that $\norm{f(x) - f(y)} \leq \lambda \norm{x - y}$ for all $x, y \in X$.
\end{defn}

If $\Eps \searrow 0$, then $\lambda \Eps \searrow 0$ for any $\lambda \in \L^+$. Using this fact, it can be shown that any Lipschitz continuous function is uniformly continuous.

\begin{defn}
Let $X$ and $Y$ be $\L$-normed spaces. An $\L$-linear operator $T \colon X \to Y$ is called \textit{bounded} if there exists a $M \in \L^+$ such that, for all $x \in X$, $\norm{Tx} \leq M \norm{x}$. In such a case we define the \textit{norm} of $T$ as
\[
        \norm{T} = \inf \{ M \in \L^+: \norm{Tx} \leq M \norm{x} \text{ for all } x \in X \}.
\]
\end{defn}

As usual, a bounded $\L$-linear operator $T \colon X \to Y$ is \emph{contractive} if $\norm{T} \leq 1$. The set of all bounded $\L$-linear operators from $X$ to $Y$ is denoted $B(X, Y)$. It is straightforward to verify that $B(X, Y)$ is an $\L$-module when equipped with the pointwise operations and that $\norm{\cdot}\colon B(X, Y) \to \L^+$ is indeed a norm in the sense of \Cref{defn: l-norm and l-seminorm}.

If we take $X = Y$, we denote the set of all bounded $\L$-linear operators $T\colon X\to X$ as $B(X)$. In the special case of $Y = \L$, we denote set of all bounded $\L$-linear operators $\varphi: X\to \L$ as $X^\ast$, which we call the \emph{dual space} of $X$. 

\begin{theorem}\label{t:op_cont_bound_equiv}
   Let $X$ and $Y$ be $\L$-normed spaces and $T \colon X \to Y$ linear. Then the following are equivalent:
   \begin{enumerate}
        \item $T$ is continuous;
        \item $T$ is continuous at $\0$;
        \item $T$ is bounded;
        \item $T$ is Lipschitz continuous;
        \item $T$ is uniformly continuous.
   \end{enumerate}
\end{theorem}
\begin{proof}

It is clear that (iv) $\Rightarrow$ (v) $\Rightarrow$ (i) $\Rightarrow$ (ii), and (iii) $\Rightarrow$ (iv) is immediate using linearity.

To show (ii) $\Rightarrow$ (iii), suppose $T$ is not bounded. Then, for each $(n, \lambda) \in \N \times \L^+$ (coordinatewise order), there exists a $y_{n, \lambda} \in B_X$ such that $\norm{T(y_{n, \lambda})} \not\leq n \lambda$. Define $x_{n, \lambda} = n^{-1}y_{n, \lambda}$,
then it is clear that $x_{n, \lambda} \to \0$ since $x_{n, \lambda} \in \frac{1}{n} B_X$. Let $\lambda_0 \in \L^+$. Then
\[ (n, \lambda) \geq (1, \lambda_0) \Rightarrow \norm{T(x_{n, \lambda})} \not\leq \lambda. \]
Any tail of $(T(y_{n, \lambda}))$ contains points indexed by $(n, \lambda) \geq (1, \lambda_0)$, so that tail is not bounded by $\lambda_0$. Since $\lambda_0$ was arbitrary, $(T(y_{n, \lambda}))$ has no bounded tails and therefore does not converge by Proposition \ref{p:conv_net_bounded_tail}. Hence $T$ is not continuous at $\0$.
\end{proof}

The fact that bounded operators are continuous is used in the following familiar alternative characterisation of the norm of an operator.

\begin{theorem}\label{t:op_norm_equiv}
Let $X$ and $Y$ be $\L$-normed spaces with $T \colon X \to Y$ a bounded $\L$-linear operator. Then
\[
        \norm{T} = \sup_{x \in B_X} \norm{Tx}.
\]
\end{theorem}

\begin{proof}
    For any $x \in B_X$, $ \norm{Tx} \leq \norm{T} \norm{x} \leq \norm{T}.$
    Therefore, $\sup_{x \in B_X} \norm{Tx} \leq \norm{T}$. For the other inequality, let $x \in X$. We first consider the case where $X$ is complete. Then $n_x \in B_X$ and so by \Cref{p: normalise}~(ii), 
    \begin{equation}\label{e:op_norm_equiv_estimate}
    \norm{Tx} = \norm{ T\left( \norm{x}n_x  \right) } = \norm{x} \norm{ T(n_x)} \leq \left( \sup_{x \in B_X} \norm{Tx} \right) \norm{x}.
    \end{equation}
    If $X$ is not complete, instead apply a similar argument to $\dis{x_n \defeq \frac{\norm{x}}{\norm{x} + n^{-1}} x}$ obtaining $\norm{Tx_n} \leq \sup_{x \in B_X} \norm{Tx} \norm{x}$, and since $x_n \to x$, taking the limit for $n \to \infty$ yields \eqref{e:op_norm_equiv_estimate}. Hence $\norm{T} \leq \sup_{x \in B_X} \norm{Tx}$.
\end{proof}

\begin{defn}
Let $X$ and $Y$ be $\L$-normed spaces. An $\L$-linear operator $T\in B(X, Y)$ is an \emph{isomorphism} if $T$ is bijective and both $T$ and $T^{-1}$ are bounded, and $T$ is an \emph{isometry} if $\norm{Tx} = \norm{x}$ for all $x \in X$. 
\end{defn}

An $\L$-linear operator $T\in B(X, Y)$ is an isometric isomorphism if and only if $T$ is bijective and both $T$ and $T^{-1}$ are contractive.

\begin{theorem}\label{t:Y-complete}
Consider $\L$-normed spaces $X$ and $Y$. If $\ Y$ is complete, then so is $B(X,Y)$. 
\end{theorem}

\begin{proof}
Let $(T_\alpha)$ be a Cauchy net in $B(X,Y)$. Then $(T_\alpha x)$ is Cauchy in $Y$ and hence converges to some $Tx \in Y$. The continuity of addition and scalar multiplication imply that $T \in L(X,Y)$. Note that $(T_\alpha|_{B_X})$ is a Cauchy net in $\ell^\infty(B_X, Y)$, which by \Cref{t:completeness_bounded_functions} converges uniformly, and it is clear that the uniform limit is also the pointwise limit $T|_{B_X}$. Thus $T \in B(X,Y)$ and $T_\alpha \to T$ in the operator norm.
\end{proof}

\begin{corol}
For every $\L$-normed space $X$, the dual space $X^\ast$ is an $\L$-Banach space.     
\end{corol}


The proof of the following proposition is the same as that of classical case.
\begin{prop}\label{p:operators_form_banach_algebra}
Let $X$, $Y$, and $Z$ be $\L$-normed spaces. If $T \in B(X, Y)$ and $S \in B(Y, Z)$ then $ST \in B(X, Z)$ and $\norm{ST} \leq \norm{S}\norm{T}$.
\end{prop}


From \Cref{t:Y-complete} and \Cref{p:operators_form_banach_algebra}, we obtain the following corollary. 
\begin{corol}
    If $X$ is an $\L$-normed space, then $B(X)$ is an $\L$-normed algebra. If additionally $X$ is complete, $B(X)$ is an $\L$-Banach algebra.
\end{corol}

\subsection{Convexity}

\begin{defn}\label{defn:l-convex}
    Let $X$ be an $\L$-vector space. A subset $C \subseteq X$ is called \emph{$\L$-convex} if $\l x + \mu y \in C$ for all $x,y \in C$ and $0 \leq \l , \mu \in \L$ such that $\l + \mu = 1$. $C$ is called \emph{$\idempotents$-convex} if $\pi x + \pi^c y \in C$ for all $x, y \in C$ and $\pi \in \idempotents$.
\end{defn}

In the classical case, any subset of a vector space is $\idempotents$-convex. The idea of the next proof is similar to the proof of \cite[Lemma~1.13]{edeko2023decomposition}.




\begin{lemma} \label{l:zero_distance_convex}
Let $X$ be an $\L$-normed space with $C \subseteq X$ $\idempotents$-convex and $x \in X$. Then there exists a net $(y_\alpha)$ in $C$ such that $\norm{y_\alpha - x} \downarrow \inf_{y \in C} \norm{y - x}$. In particular, if $\inf_{y \in C} \norm{y - x} = 0$ then $y_\alpha \to x$.
\end{lemma}
\begin{proof}
Define $\preceq$ on $C$ by $y_1 \preceq y_2$ if and only if $\norm{x - y_2} \leq \norm{x - y_1}$. We show that this turns $C$ into a directed set. Let $\l = \norm{y - x}$ and $\mu = \norm{z - x}$ for some $y, z \in C$. Let $\pi := \pi_{(\l-\mu)^-}$. Set $v := \pi y + \pi^c z \in C$. We have
\begin{align*}
    \norm{x-v}&=\norm{\pi x +\pi^c x- \pi y- \pi^c z}\\
    &= \pi\norm{x-y} +\pi^c\norm{x- z}\\
    &= \pi\l +(1-\pi)\mu\\
    &= \pi(\l-\mu) + \mu\\
    &=(\l-\mu) \wedge 0 + \mu\\
    &=\l\wedge\mu.
\end{align*}
Thus, $y, z \preceq v$ and hence $C$ is directed. Now the decreasing net $(\norm{x-y})_{y \in C}$ converges to its infimum, completing the proof.
\end{proof}

We use this lemma to show that the closure of $\idempotents$-convex subsets equals the set of adherent points.

\begin{corol}\label{c:closure_equals_adherence}
Let $X$ be an $\L$-normed space with $C \subseteq X$ $\idempotents$-convex. Then, 
\[ \overline{C} = \{ x \in X: \text{there is a net } (x_\alpha) \text{ in } C \text{ such that } x_\alpha \to x \}. \]
\end{corol}
\begin{proof}
Fix $x \in \overline{C}$. We want that $\inf_{y \in C} \norm{x - y} = 0$, as then there is a net $(x_\alpha)$ in $C$ with $x_\alpha \to x$ by \Cref{l:zero_distance_convex}. It suffices to show that $N \defeq \{z \in X: \inf_{y \in C} \norm{z - y} = 0 \}$ is closed since it clearly contains $C$ and must then contain $\overline{C}$. Suppose $(z_\alpha)$ is in $N$ with $z_\alpha \to z$. Then for any $y \in C$ and any $\alpha$,
\[ \norm{z - y} \leq \norm{z - z_\alpha} + \norm{z_\alpha - y}. \]
Taking infima over $y \in C$ yields $\inf_{y \in C} \norm{z - y} \leq \norm{z - z_\alpha}$ for each $\alpha$. Thus $z \in N$, making $N$ closed.

The reverse inclusion is easy to see as $\overline{C}$ would fail to be closed if it were false.
\end{proof}

Using \Cref{c:closure_equals_adherence}, the proof of the next proposition is as in the classical case. 
\begin{prop}
Let $X$ be an $\L$-normed space with dense subspace $X_0$ and let $Y$ be an $\L$-Banach space. Then each $T \in B(X_0, Y)$ has a unique extension to an element $\hat{T} \in B(X, Y)$.  If $T$ is isometric, then so is $\hat{T}$. The mapping $T \mapsto \hat{T}$ is an isometric isomorphism.
\end{prop}

\subsection{Quotient spaces}

\begin{prop}\label{p:rho_seminorm}
    Let $X$ be an $\L$-normed space and let $Y \subseteq X$ be a submodule. Define $\rho \colon X \to \L^+$ by $\rho(x) := \inf \{\norm{x-y} \colon y \in Y\}$.
    Then $\rho$ is an $\L$-seminorm with kernel $\overline{Y}$.
\end{prop}
\begin{proof}
    Let $\l \in \L$. By \Cref{l:approximate_with_invertibles}, there exists a sequence of invertible elements $\l_n \to \l$. Then
    \begin{align*}
        \rho(\l_n x) &= \inf \{ \norm{\l_n x - y} \colon y \in Y\} \\&= \inf \{ |\l_n| \norm{x - \l_n^{-1} y} \colon y \in Y\} \\
        &= |\l_n| \inf \{ \norm{x -  y} \colon y \in Y\} \\&= |\l_n| \rho(x) \to |\l| \rho(x). 
    \end{align*}
    On the other hand, for $y \in Y$, $\norm{\l x - y} \leq \norm{\l_n x - y} + |\l - \l_n| \norm{x}$ and $\norm{\l_n x - y} \leq \norm{\l x - y} + |\l_n - \l| \norm{x}$. It follows that \[
    \rho(\l x) \leq \rho( \l_n x) + |\l-\l_n| \norm{x} \quad \mbox{and} \quad \rho(\l_n x) \leq \rho( \l x) + |\l_n-\l| \norm{x},
    \]
    hence $|\rho(\l x) - \rho(\l_n x)| \leq |\l - \l_n| \norm{x} \to 0$, and so $\rho(\l_n x) \to \rho(\l x)$, thus by uniqueness of limits, $\rho(\l x) = |\l| \rho(x)$.
    Let $x, z \in X$ and $y \in Y$, then
    \[
    \norm{(x+z) - y} \leq \norm{x - {\textstyle \frac{1}{2}}y} + \norm{z - {\textstyle \frac{1}{2}}y}
    \]
    and so $\rho(x+z) \leq \rho(x) + \rho(z)$, showing that $\rho$ is a seminorm.

    By \Cref{l:zero_distance_convex} and \Cref{c:closure_equals_adherence}, $x \in \ker(\rho)$ if and only if there is a net $(y_\alpha)$ in $Y$ such that $y_\alpha \to x$ if and only if $x \in \overline{Y}$.   
\end{proof}
The next corollary is now immediate.
\begin{corol}
    Let $X$ be an $\L$-normed space and let $Y \subseteq X$ be a submodule. On $X/Y$, define $\rho(x + Y) := \inf \{ \norm{x-y} \colon y \in Y\}$. Then $(X/Y, \rho)$ is an $\L$-seminormed space. Moreover, $X/Y$ is an $\L$-normed space if and only if $Y$ is closed.
\end{corol}

\section{Spaces of $p$-summable functions}\label{s:l^p spaces}

The following lemma is easy but we have not been able to find a reference for it.
\begin{lemma}\label{l:r inequality}
Let $t, s \in \R^+$ and $0 \leq r \leq 1$. Then $|t^r-s^r| \leq |t-s|^r$.
\end{lemma}
\begin{proof}
    Suppose $0 < s < t$ (all other cases are trivial). Then $1 - (s/t)^r \leq 1 - s/t$ and $(1 - s/t)^r \geq 1 - s/t$, so
    \[
    \frac{t^r-s^r}{(t-s)^r} = \frac{1 - (\frac{s}{t})^r}{ \left( 1 - \frac{s}{t} \right)^r  } \leq \frac{1 - \frac{s}{t}}{1 - \frac{s}{t}} = 1. \qedhere
    \]
\end{proof}

\begin{lemma}\label{l:p power continuous}
Let $0 < p < \infty$. Then
\begin{enumerate}
    \item If $\Eps \searrow 0$, then $\Eps^p \searrow 0$;
    \item The map $\lambda \mapsto \lambda^p$ from $\L^+$ to $\L^+$ is continuous.
\end{enumerate}
\end{lemma}
\begin{proof}
    (i): Let $\Eps \searrow 0$ and suppose $\l$ is a lower bound of $\Eps^p$. Then for all $\eps \in \Eps$, we have $0 \leq \l^+ \leq \eps^p$, and so $0 \leq (\l^+)^{1/p} \leq \eps$, implying that $(\l^+)^{1/p} = 0$. Thus $\l^+ = 0$ and so $\l \leq 0$, hence $\inf(\Eps^p) = 0$.
    
    (ii): Let $n \in  \N_0$ and $0 \leq r < 1$ be such that $p = n+r$, then $\l^p = \l^n \l^r$ and it suffices to prove continuity of $\l \mapsto \l^r$, which is trivial if $r=0$. So let $0 < r < 1$ and assume $\l_\alpha \to \l$, so that there exists a set $\Eps \searrow 0$ with $|\l_\alpha - \l| \leq \eps$ for all $\eps \in \Eps$. By \Cref{l:r inequality} and representation theory, 
    \[
    |\l_\alpha^r - \l^r| \leq |\l_\alpha - \l|^r \leq \eps^r.
    \]
    Thus by (i), the set $\Eps^r \searrow 0$ witnesses the convergence of $\l_\alpha^r$ to $\l^r$.
\end{proof}


The following versions of the H\"{o}lder and Minkowski inequalities follow directly by applying the classical versions of these inequalities pointwise in a representation theory argument. Alternatively, these results follow from the more general theorems \cite[Theorem~$4.7$]{BusSch3} and \cite[Theorem~$5.1$]{BusSch3}, respectively. Recall that $1 \leq p, q \leq \infty$ are called \emph{conjugate indices} if they satisfy $1/p + 1/q = 1$, using the convention $1 / \infty = 0$.


\begin{theorem}[H\"{o}lder inequality]\label{t: Holder}
Consider $(\l_k)_{k=1}^n, (\mu_k)_{k=1}^n\in \L^n$ for some $n \in \N$. Then for conjugate indices $1 \leq p,q \leq \infty$  we have
\[
        \sum_{k=1}^n\left| \l_k \mu_k \right| \leq \left(\sum_{k=1}^n\left| \l_k \right|^p  \right)^{1/p}\left( \sum_{k=1}^n\left| \mu_k \right|^q 
 \right)^{1/q}.
\]
\end{theorem}


\begin{theorem}[Minkowski inequality]\label{t:Minkowski}
Consider $(\l_k)_{k=1}^n, (\mu_k)_{k=1}^n\in \L^n$ for some $n \in \N$. Then for $1 \leq p < \infty$  we have
\[
        \left( \sum_{k=1}^n\left| \l_k + \mu_k \right|^p  \right)^{1/p} \leq \left( \sum_{k=1}^n\left| \l_k \right|^p  \right)^{1/p} + \left( \sum_{k=1}^n\left| \mu_k \right|^p  \right)^{1/p}.
\]
\end{theorem}

 
For a set $I$, denote by $\mathcal{F}(I)$ the collection of all finite subsets of $I$. Let $X$ be an $\L$-normed space with $(x_i)_{i\in I}$ a collection of elements in $X$. The collection of finite sums
\[
        \left( \sum_{i\in F}x_i \right)_{F\in \mathcal{F}(I)}
\] forms a net in $X$. If this net is convergent, we express the limit as the formal series $\sum_{i \in I}x_i$ and we say that the series $\sum_{i \in I}x_i$ is \emph{convergent}. We note here for an arbitrary subset $J \subseteq I$, if $\sum_{i \in I}x_i$ is convergent, then
\[
        \sum_{i\in I}x_i = \sum_{i\in J}x_i + \sum_{i\in J^c}x_i.
\]Further, we say that the series $\sum_{i \in I}x_i$ is \emph{absolutely convergent} if the net
\[
        \left( \sum_{i\in F}\norm{x_i}_X \right)_{F\in \mathcal{F}(I)}
\]is convergent in $\L$. 

\begin{theorem}\label{t: Absolute convergence implies convergence}
Let $X$ be an $\L$-Banach space. If a series in $X$ converges absolutely, then it converges.
\end{theorem}

\begin{proof}
Let $\sum_{i\in I}x_i$ be an absolutely convergent series, then the net $\left( \sum_{i\in F} \norm{x_i} \right)$ is Cauchy. Thus there exists an $\Eps \searrow 0$ such that for every $\eps \in \Eps$ there exists an $F_0$ such that for all finite $F, G \supseteq F_0$,
\[ \abs{ \sum_{i \in F} \norm{x_i} - \sum_{i \in G} \norm{x_i} } \leq \eps. \]
Fix an $\eps \in \Eps$ and its corresponding $F_0$. Then for all finite $F, G \supseteq F_0$,
\begin{align*}
\norm{ \sum_{i \in F} x_i - \sum_{i \in G} x_i}
& = \norm{\sum_{i \in F \setminus G} x_i - \sum_{i \in G \setminus F} x_i} \\
& \leq \sum_{i \in F \setminus G} \norm{x_i} + \sum_{i \in G \setminus F} \norm{x_i} \\
& = \sum_{i \in F} \norm{x_i} - \sum_{i \in F \cap G} \norm{x_i}\ + \sum_{i \in G} \norm{x_i} - \sum_{i \in F \cap G} \norm{x_i} \\
& \leq 2 \eps.
\end{align*}
Therefore, the net $\left( \sum_{i \in F} x_i \right)$ is Cauchy, so $\sum_{i \in I} x_i$ converges.
\end{proof}

\begin{example}
Let $S$ be a non-empty set, let $Y$ be an $\L$-normed space, and let $1 \leq p < \infty$. We denote by $\ell^p\left( S, Y \right)$ the set of functions $f \colon S \to Y$ for which
\[
        \norm{f}_p \defeq \left( \sum_{s\in S} \norm{f(s)}^p_Y \right)^{1/p}
\] exists in $\L$. For $f\in \ell^p\left( S, Y \right)$, since the net of finite sums $\left( \sum_{s\in F} \norm{f(s)}^p_Y \right)_{F\in \mathcal{F}(S)}$ is increasing and bounded by $\norm{f}^p_p$, the $p$-norm can also be expressed as
\[
        \norm{f}_p = \left( \sup_{F \in \mathcal{F}(S)} \sum_{s\in F} \norm{f(s)}^p_Y \right)^{1/p} = \sup_{F \in \mathcal{F}(S)} \left( \sum_{s\in F} \norm{f(s)}^p_Y \right)^{1/p}
\]where the second equality follows by the continuity of exponentiation. We now show that ${\norm{ \cdot }_p: \ell^p\left( S, Y \right) \to \L^+}$ is indeed an $\L$-norm. Absolute homogeneity follows from 
 \Cref{l:order_cont_mult} since for every $f\in \ell^p\left( S, Y \right)$ and $\l \in \L$ we have
\[
        \norm{ \l f }^p_p = \sup_{F \in \mathcal{F}(S)} \sum_{s\in F} \norm{\l f(s)}^p_Y = \sup_{F \in \mathcal{F}(S)} \left| \l \right|^p \sum_{s\in F} \norm{f(s)}^p_Y =   \left| \l \right|^p \norm{ f }^p_p. 
\]To prove the triangle inequality, consider $f,g\in \ell^p\left( S, Y \right)$ and $F \in \mathcal{F}(S)$. 
By monotonicity of exponents and the Minkowski inequality (\Cref{t:Minkowski}), we have
\begin{align*}
        \norm{ f + g }_p &= \sup_{F \in \mathcal{F}(S)} \left(\sum_{s\in F} \norm{f(s) + g(s)}^p_Y \right)^{1/p}\\
                         &\leq \sup_{F \in \mathcal{F}(S)}\left( \sum_{s\in F} \norm{f(s)}^p_Y \right)^{1/p} + \sup_{F \in \mathcal{F}(S)} \left( \sum_{s\in F} \norm{g(s)}^p_Y \right)^{1/p}\\
                         &= \norm{f}_p + \norm{g}_p.
\end{align*}
Since $\norm{f}_p = 0$ implies that $f(s) = 0$ for all $s\in S$, we conclude that $\ell^p\left( S, Y \right)$ is an $\L$-normed space. In particular, if $S$ is a finite set with $n$ elements, then the above argument shows that $( Y^n, \norm{\cdot}_p )$ is an $\L$-normed space.
\end{example}
              
\begin{theorem}
Let $S$ be a non-empty set, let $Y$ be an $\L$-Banach space, and let $1 \leq p < \infty$. Then $\ell^p\left( S, Y \right)$ is an $\L$-Banach space.
\end{theorem}

\begin{proof}
Let $(f_\alpha)$ be a Cauchy net in $\ell^p(S, Y)$. Then there exists $\Eps \searrow 0$ such that for all $\eps \in \Eps$ there exists an $\alpha_0$ satisfying
\begin{equation}\label{e:lp-cauchy}
        \alpha, \beta \geq \alpha_0 \Rightarrow \norm{f_\alpha - f_\beta}_p \leq \eps.
\end{equation}Since $\norm{f_\alpha(s) - f_\beta(s)}_Y \leq \norm{f_\alpha - f_\beta}_p$ for every $s\in S$, it follows that $(f_\alpha(s))$ is Cauchy in $Y$ for every $s \in S$. Since $Y$ is complete, this defines a mapping $f\colon S\to Y$ where $f(s) \defeq \lim_{\alpha}f_\alpha(s)$. It remains to show that $f \in \ell^p(S, Y)$ and that $f_\alpha \to f$ in $\ell^p(S, Y)$. By monotonicity of exponents, for every $s\in S$ and any $\alpha$, we have 
\[
        \norm{f(s)}^p_Y \leq \big| \norm{f(s) - f_\alpha(s)}_Y + \norm{f_\alpha(s)}_Y \big|^p. 
\]Using the above inequality and the Minkowski inequality, for every $F \in \mathcal{F}(S)$ and any $\alpha$, we find
\begin{align}\label{e: Estimate for f being in l^p}
        \nonumber\left( \sum_{s \in F}\norm{f(s)}^p_Y \right)^{1/p} &\leq \left( \sum_{s \in F}\norm{f(s) - f_\alpha(s)}^p_Y \right)^{1/p} + \left( \sum_{s \in F}\norm{f_\alpha(s)}^p_Y \right)^{1/p}\\
                                                           &\leq  \left( \sum_{s \in F}\norm{f(s) - f_\alpha(s)}^p_Y \right)^{1/p} + \norm{f_\alpha}_p. 
\end{align}
Fix $\eps \in \Eps$ and choose $\alpha_0
$ as in \eqref{e:lp-cauchy}. Then for all $F \in \mathcal{F}(S)$, we have
\[
        \alpha,\beta \geq \alpha_0 \Rightarrow \sum_{s\in F
} \norm{f_\alpha(s) - f_\beta(s)}^p_Y \leq \eps^p.
\]Since $f_\beta(s) \to f(s)$ in $Y$ for all $s\in S$, it follows that $f_\alpha(s) - f_\beta(s) \to f_\alpha(s) - f(s)$ for all $\alpha \geq \alpha_0$ and for all $s\in S$. By the continuity of the norm, as well as addition and exponentiation in $\L$, it follows that 
\[
        \sum_{s\in F} \norm{f_\alpha(s) - f_\beta(s)}^p_Y \to \sum_{s\in F} \norm{f_\alpha(s) - f(s)}^p_Y,
\] for all $F \in \mathcal{F}(S)$ and $\alpha \geq \alpha_0$. Since the set $ \{\lambda \in \L ~:~ \lambda \leq \eps^p \}$ is closed, we conclude that for all $F \in \mathcal{F}(S)$ and $\alpha \geq \alpha_0$, 
\begin{align}\label{e: f_alpha converges to f in the p-norm}
         \left( \sum_{s \in F}\norm{f(s) - f_\alpha(s)}^p_Y \right)^{1/p} \leq \eps,
\end{align}  
which we may use along with \eqref{e: Estimate for f being in l^p} to conclude that $f \in \ell^p(S, Y)$ and \eqref{e: f_alpha converges to f in the p-norm} implies that $f_\alpha \to f$ in $\ell^p(S, Y)$.
\end{proof}

The rest of this section is devoted to proving that for $1\leq p < \infty$, the dual of $\ell^p(S,Y)$ is isometrically isomorphic to $\ell^q(S,Y^*)$ where $1< q \leq \infty$ is the conjugate index of $p$. As in the classical case, the proof that the dual of $c_0(S,Y)$ is isometrically isomorphic to $\ell^1(S,Y^*)$ proceeds similarly and will therefore be included as well. First, we note the important fact that $\L$ is self-dual.

\begin{remark}\label{r: L self dual}
For every $\mu \in \L$, we associate the $\L$-linear multiplication operator $m_\mu: \L \to \L$ given by $m_\mu(\l)\defeq\mu \l$. By \Cref{t:op_norm_equiv}, we have $\norm{m_\mu} = |\mu|$. Conversely, if $\phi\colon \L \to \L$ is $\L$-linear and $\l \in \L$, then $\phi(\l) = \phi(\l 1) = \l \phi(1)$ so it is clear that $\phi \mapsto \phi(1)$ is the inverse correspondence.  Hence $\L$ is isometrically isomorphic to $\L^*$.
\end{remark}

Let $S$ be a non-empty set and let $Y$ be an $\L$-normed space. Denote by $c_{00}(S, Y)$ the subspace of functions $g\colon S \to Y$ in $\ell^\infty(S,Y)$ where $\{ s\in S : g(s) \neq 0 \}$ is finite. Further, given any function $f\colon S\to Y$ and any subset $F\subseteq S$, we define $P_F f$, the projection of $f$ onto $F$, by $(P_F f)(s) \defeq f(s)$ for $s \in F$ and $(P_F f)(s) \defeq \0$ for $s \in F^c$. It is clear that for any function $f\colon S\to Y$ and any $F\in \mathcal{F}(S)$ that $P_F f \in c_{00}(S, Y)$. With this notation, we can prove the following crucial lemma. 

\begin{lemma}\label{l: Approximation of elements in l^p via sum}
Let $S$ be a non-empty set and $Y$ an $\L$-normed space. For $f \in \ell^p(S, Y)$ with $1 \leq p < \infty$ or $f \in c_0(S, Y)$, the net $\left( P_F f \right)_{F\in \mathcal{F}(S)}$ converges to $f$. Thus $c_{00}(S,Y)$ is dense in $\ell^p(S, Y)$ and $c_0(S, Y)$.
\end{lemma}

\begin{proof}
Fix $1 \leq p < \infty$ and consider $f \in \ell^p(S, Y)$. Then there exists an $\Eps \searrow 0$ such that for every $\eps \in \Eps$ there exists an $F_0 \in \mathcal{F}(S)$ such that for $F \supseteq F_0$,
\[
        \sum_{s \in F^c} \norm{f(s)}_Y^p = \sum_{s \in S} \norm{f(s)}_Y^p - \sum_{s \in F} \norm{f(s)}_Y^p \leq \eps.
\]
Fix an $\eps \in \Eps$ and its corresponding $F_0$. Now, for $F \supseteq F_0$,
\[
        \norm{f - P_F f}_p^p = \sum_{t \in S} \norm{ \left( f - P_F f \right) (t)}_Y^p = \sum_{t \in F^c} \norm{f(t)}_Y^p \leq \eps,
\]
so that $\norm{f - P_F f} \leq \eps^{\frac{1}{p}}$. The claim follows since $\Eps^{\frac{1}{p}} \searrow 0$.

Now suppose $f \in c_0(S, Y)$. Then there exists an $\Eps \searrow 0$ such that for all $\eps \in \Eps$ there exists a cofinite $C_\eps \subseteq S$ such that for all $s \in C_\eps$, $\norm{f(s)}_Y \leq \eps$. For each $F \in \mathcal{F}(S)$, $\norm{f - P_F f}_\infty = \sup_{s \in F^c} \norm{f(s)}_Y$; so, it is clear that the net $(\norm{f - P_F f}_\infty)_{F \in \mathcal{F}(S)}$ is decreasing. Now, for any $\eps \in \Eps$ and $F \supseteq C_\eps^c$,
\[ \norm{f - P_F f}_\infty \leq \norm{f - P_{C_\eps^c} f}_\infty = \sup_{s \in C_\eps} \norm{f(s)}_Y \leq \eps. \]
Therefore, $P_F f \to f$ in $c_0(S,Y)$.
\end{proof}

\begin{remark}
The notation in the previous result can be reconciled with the formal series notation introduced earlier in this section in the following way: Let $g\colon S\to \L$ be any function and fix $y\in Y$. Define $g\otimes y\colon S\to Y$ by $\left( g\otimes y\right)(s) \defeq g(s) y$. Now, for $s\in S$ define $e_s: S\to \L$ by $e_s(t) = \delta_{st}$. For any function $h\colon S\to Y$ and $F \in \mathcal{F}(S)$ it is clear that 
\[
        P_F h = \sum_{s \in F} e_s \otimes h(s). 
\]In particular, for a function $f\in \ell^p(S, Y)$ where $1 \leq p < \infty$ or for $f\in c_0(S, Y)$, the fact that the net of projections $\left( P_F f \right)_{F\in \mathcal{F}(S)}$ converges to $f$ can then be expressed succintly as 
\[
        f = \sum_{s \in S} e_s \otimes f(s).
\]
\end{remark}
We will make use of the above notation in the following result. Additionally, for $y\in Y$ and $y^*\in Y^*$, we define $\ip{y}{y^*}$ to be $y^*(y)$.

\begin{theorem}
Let $S$ be a non-empty set, $Y$ be an $\L$-normed space, and $1 \leq p < \infty$ with $q$ its conjugate index. Then the dual space of $\ell^p(S,Y)$ is isometrically isomorphic to $\ell^q(S, Y^*)$, implemented by the $\L$-bilinear map $B \colon \ell^p(S,Y) \times \ell^q(S, Y^*) \to \L$ defined by
\[
        B(f,g) \defeq \sum_{s\in S} \ip{f(s)}{g(s)}.        
\]
Furthermore, $c_0(S,Y)^* \cong \ell^1(S, Y^*)$ through the same duality $B$.
\end{theorem}

\begin{proof}
Note that the existence of $B(f,g) \in \L$ for $(f,g) \in \ell^p(S,Y) \times \ell^q(S,Y^*)$ follows from the H\"{o}lder inequality (\Cref{t: Holder}) since for every $F\in \mathcal{F}(S)$, we have
\begin{align}\label{e: Application of Holder}
         \sum_{s\in F}\left| \ip{f(s)}{g(s)}  \right| \leq \sum_{s\in F} \norm{f(s)}_{Y} \norm{g(s)}_{Y^*} \leq \norm{f}_p \norm{g}_q. 
\end{align}
Thus the formal sum $\sum_{s\in S} \ip{f(s)}{g(s)}$ is absolutely convergent, hence convergent by \Cref{t: Absolute convergence implies convergence}. The same argument is valid for $(f,g) \in c_0(S,Y) \times \ell^1(S,Y^*)$.

Let $g \in \ell^q(S,Y^*)$, define the $\L$-linear map $\phi_g\colon \ell^p(S,Y) \to \L$ by $f \mapsto B(f,g)$. Since the net $\left( \left|  \sum_{s\in F}\ip{f(s)}{g(s)} \right| \right)_{F\in \mathcal{F}(S)}$ converges to $\left| B(f,g) \right|$ and since $\left|  \sum_{s\in F}\ip{f(s)}{g(s)} \right| \leq  \norm{f}_p \norm{g}_q$ for every $F\in \mathcal{F}(S)$ we have
\[
        \left| \phi_g\left( f \right) \right| \leq \norm{f}_p \norm{g}_q. 
\]Thus $\phi_g \in \ell^p(S,Y)^*$ and $\norm{\phi_g} \leq \norm{g}_q$. The same argument shows that for $g \in \ell^1(S,Y^*)$, we have that $\phi_g \in c_0(S,Y)^*$ and $\norm{\phi_g} \leq \norm{g}_1$.

Take $\phi \in \ell^p(S,Y)^*$ and define $g_\phi\colon S \to Y^*$ by $\ip{y}{g_\phi(s)} \defeq \phi(e_s \otimes y)$. We will show that $g_\phi \in \ell^q(S, Y^*)$. If $p=1$, then for every $s \in S$, we have
\begin{align*}
        \norm{g_\phi(s)}_{Y^*} = \sup_{y\in B_Y} \left| \ip{y}{g_\phi(s)} \right| = \sup_{y\in B_Y} \left| \phi(e_s \otimes y) \right| \leq \sup_{y\in B_Y} \norm{\phi} \norm{e_s \otimes y}_1 \leq \norm{\phi}.      
\end{align*}
This shows that $g_\phi \in \ell^\infty(S, Y^*)$ and $\norm{g_\phi}_\infty \leq \norm{\phi}$. If $p > 1$, consider $F\in \mathcal{F}(S)$. For $s\in F$, let $y_s \in B_Y$ and denote by $\gamma_{y_s} \in \L$ the element defined in \Cref{c: Rotation by scalar multiplication} such that $\gamma_{y_s} \ip{y_s}{g_\phi(s)} = \left| \ip{y_s}{g_\phi(s)} \right|$ and $|\gamma_{y_s}| \leq 1$. Then
\begin{align*}
    \sum_{s\in F}  \left| \ip{y_s}{g_\phi(s)} \right|^q =\ &\sum_{s\in F}  \gamma_{y_s} \ip{y_s}{g_\phi(s)} \left| \ip{y_s}{g_\phi(s)} \right|^{q-1}\\
    =\ &\sum_{s\in F} \gamma_{y_s} \phi(e_s \otimes y_s)  \left| \ip{y_s}{g_\phi(s)} \right|^{q-1}\\
    =\ &\phi\left( \sum_{s\in F} \gamma_{y_s}  \left| \ip{y_s}{g_\phi(s)} \right|^{q-1} e_s \otimes y_s \right)\\
    \leq\ &\norm{\phi} \norm{\sum_{s\in F} \gamma_{y_s}  \left| \ip{y_s}{g_\phi(s)} \right|^{q-1} e_s \otimes y_s}_p\\
    \leq\ & \norm{\phi} \left( \sum_{s\in F} \left| \ip{y_s}{g_\phi(s)} \right|^{q} \right)^{1/p}
\end{align*}
If $\l, \mu \in \L^+$, then using representation theory it is easy to see that if $\l \leq \mu \l^{1/p}$, then $\l^{1- 1/p} \leq \mu$. Applying this implication to the case where $\l = \sum_{s\in F} \left|\ip{y_s}{g_\phi(s)} \right|^q$ and $\mu = \norm{\phi}$ and since $1 -1/p = 1/q$, we conclude that
\[
        \left( \sum_{s\in F} \left|\ip{y_s}{g_\phi(s)} \right|^q \right)^{1/q}  \leq \norm{\phi}. 
\]It now follows by continuity of addition and exponentiation that
\[
      \left( \sum_{s\in F}\norm{g_\phi(s)}_{Y^*}^q \right)^{1/q} = \sup_{\substack{y_s \in B_Y \\ s\in F}} \left(\sum_{s\in F} \left|\ip{y_s}{g_\phi(s)} \right|^q \right)^{1/q} \leq \norm{\phi}.
\]
Thus $g_\phi \in \ell^q(S, Y^*)$ and $\norm{g_\phi}_q \leq \norm{\phi}$. 

If $\phi \in c_0(S,Y)^*$, define $g_\phi$ as above. Take $F \in \mathcal{F}(S)$ and for $s \in F$, let $y_s \in B_Y$. Then  
\begin{align*}
    \sum_{s\in F} \left| \ip{y_s}{g_\phi(s)} \right| =\ &\sum_{s\in F} \gamma_{y_s} \ip{y_s}{g_\phi(s)}\\
    =\ &\sum_{s\in F} \gamma_{y_s} \phi(e_s \otimes y_s)\\
    =\ &\phi\left( \sum_{s\in F} \gamma_{y_s} e_s \otimes y_s \right)\\
    \leq\ &\norm{\phi} \norm{\sum_{s\in F} \gamma_{y_s} e_s \otimes y_s}_\infty. 
\end{align*}
Since $|\gamma_{y_s}| \leq 1$ and $y_s \in B_Y$ for $s\in F$, we conclude that $g_\phi \in \ell^1(S,Y^*)$ and $\norm{g_\phi}_1 \leq \norm{\phi}$.

We show that the correspondences $g \mapsto \phi_g$ and $\phi \mapsto g_\phi$ are inverses of each other. Let $g\in \ell^q(S,Y^*)$ ($1 \leq q \leq \infty$) and fix $s\in S$ and $y\in Y$, then
\begin{align*}
        \ip{y}{g_{\phi_g}(s)}  = \phi_g\left( e_s \otimes y \right) = B\left( e_s \otimes y, g \right) = \sum_{t\in S}\ip{e_s \otimes y(t)}{g(t)} = \ip{y}{g(s)}.     
\end{align*}
Thus $g_{\phi_g} = g$. Conversely, let $\phi \in \ell^p(S,Y)^*$ (the case $\phi \in c_0(S,Y)^*$ is the same). Fix $f \in \ell^p(S,Y)$, then using the continuity of $\phi$ along with \Cref{l: Approximation of elements in l^p via sum}, we have
\[
    \phi_{g_\phi}(f) = B\left( f, g_\phi \right)  = \sum_{s\in S} \phi(e_s \otimes f(s)) = \phi\left( \sum_{s\in S}e_s \otimes f(s)\right) = \phi(f). 
\]
Thus $\phi_{g_\phi} = \phi$. Since the correspondences have already been shown to be contractive, we conclude that $\ell^p(S,Y)^* \cong \ell^q(S,Y^*)$ and $c_0(S,Y)^* \cong \ell^1(S,Y^*)$.
\end{proof}

Note that all steps of the above proof are also valid for $p = \infty$, except the very last step where the density of $c_{00}(S,Y)$ in $c_0(S,Y)$ was used (\Cref{l: Approximation of elements in l^p via sum}). 

In the case where $Y = \L$, we write $\ell^p(S) \defeq \ell^p(S,\L)$ and similarly for the other spaces of functions defined above. Since $\L \cong \L^*$ (\Cref{r: L self dual}), we obtain the following corollary. 

\begin{corol}
Let $S$ be a non-empty set and $1 \leq p < \infty$ with $q$ its conjugate index. Then $\ell^p(S)^* \cong \ell^q(S)$, implemented by the $\L$-bilinear map $B \colon \ell^p(S) \times \ell^q(S) \to \L$ defined by
\[
        B(f,g) \defeq \sum_{s\in S} f(s)g(s).
\]
Furthermore, $c_0(S)^* \cong \ell^1(S)$ through the same duality $B$.
\end{corol}

In the classical case, elements of $\ell^p(S)$ and $c_0(S)$ are always supported on a countable set. The next example shows that this does not hold in our case.

\begin{example}\label{e:uncountable support}
    Let $S$ be a nonempty set and suppose $\L$ is such that $\ell^\infty(S) \subseteq \L \subseteq \ell^0(S)$ (see \Cref{r:hyperstonean}). For $F \subseteq S$, denote by $\pi_F \in \idempotents$ the indicator function of $F$, and for $s \in S$, we define $\pi_s \defeq \pi_{\{s\}}$. Define $f \colon S \to \L$ by $f(s) = \pi_s$. The computation 
    \[\norm{f}^p_p=\sum_{s\in S} |f(s)|^p = \sum_{s\in S} \pi_s = 1\]
    shows that $f\in \ell^p(S)$ for all $1\leq p< \infty$. Let $\Eps\defeq \{\pi_C : C \subseteq S \text{ is cofinite}\},$ then $\inf \Eps =0$. Let $\eps \in \Eps$, then there exists a cofinite $C\subseteq S$ with $\eps=\pi_C$, and for all $s\in C$, $|f(s)|=\pi_s\leq \pi_C=\eps$. By definition, $f\in c_0(S).$  If $S$ is uncountable, then $f$ is supported on an uncountable set.
\end{example}

\section{The Hahn-Banach theorem and dual spaces}\label{s:hahn banach}

\begin{defn}\label{d:sublinear}
Let $X$ be an $\L$-module. A map $\sigma\colon X \to \L$ is \emph{sublinear} if 
\begin{enumerate}
    \item $\sigma(\l x)=\l \sigma(x)$ for all $\l\in \L^+$ and all $x\in X$; and 
    \item $\sigma(x+y)\leq \sigma(x)+\sigma(y)$ for all $x,y\in X$.
\end{enumerate}
\end{defn}

\begin{theorem}[Hahn-Banach]\label{t:HB}
Let $X$ be a real $\L$-module and let $Y \subsetneq X$ be a submodule. Furthermore, let $\phi \colon Y \to \L$ be $\L$-linear and $\sigma \colon X \to \L$ be sublinear such that $\phi(x) \le \sigma(x)$ for all $x \in Y$. Then there is an $\L$-linear extension $\Phi$ of $\phi$ to all of $X$ such that $\Phi(x) \le \sigma(x)$ for all $x \in X$.    
\end{theorem}

\begin{proof}
Fix $z \in X\setminus Y$. We will first show that we can extend $\phi$ to $Y \oplus \L z$ such that the extension remains bounded by $\sigma$. For any $\mu \in \L$, define the map $\Phi_\mu \colon Y \oplus \L z \to \L$ by $$\Phi_\mu(x+\l z):= \phi(x)+\mu\l,$$ which are all $\L$-linear extensions of $\phi$. For any $x,y\in Y$ we have
    \begin{align*}
         \phi(x)-\phi(y)=\phi(x-y)\le \sigma(x-y)&=\sigma(x+z-z-y)\\&\le \sigma(x+z)+\sigma(-z-y),
    \end{align*}
   which implies
   \begin{equation*}
       -\sigma(-z-y)-\phi(y) \le \sigma(x+z)-\phi(x),
   \end{equation*}
   which holds for all $y$. Since $\L$ is Dedekind complete, we define the elements 
   $$\eta:= \sup_{y\in Y}\left(-\sigma(-z-y)-\phi(y)\right) \quad \text{and}\quad \xi := \inf_{x \in Y}\left( \sigma(x+z)-\phi(x)\right)$$ in $\L$. 
   Note that $\eta \le \xi$, so let $\rho \in [\eta, \xi]$, then 
    \begin{equation}\label{e:HB}
       -\sigma(-z-y)-\phi(y) \le  \rho \le  \sigma(x+z)-\phi(x)
   \end{equation}
   for all $x,y \in Y$. Now consider $\Phi_\rho$.
    
   Suppose first that $\l\in\L^+$. By \Cref{l:approximate_with_invertibles}, there is a sequence of invertible positive elements such that $\l_n \downarrow \l$. If we replace $x$ in \eqref{e:HB} by $\l_n^{-1}x$, the second inequality yields 
   \begin{align*}
   \l_n\rho &\le \sigma(x+\l_n z)-\phi(x)\\
   &\leq \sigma(x+\l z)+\sigma(\l_n z-\l z)-\phi(x) \\
   &= \sigma(x+\l z)+(\l_n-\l)\sigma(z)-\phi(x) \\
   \end{align*} 
   and by letting $n \to \infty$ we obtain $\Phi_\rho(x+\l z) =\phi(x)+\l\rho \le \sigma(x+\l z)$ for all $x \in Y$ in this case.
   Next, if we replace $y$ by $-\l_n^{-1} x$ in \eqref{e:HB}, the first inequality yields 
   \[
   \rho \geq  -\sigma(-z+\l_n^{-1} x) - \phi(-\l_n^{-1}x)=-\l_n^{-1}\sigma(-\l_n z+x)+\l_n^{-1}\phi(x),
   \]
   and using the superaddivity of $-\sigma$, we get
   \begin{align*}
       \l_n\rho&\geq -\sigma(-\l_n z +x)+\phi(x)\\
       &=-\sigma(-\l_n z+\l z +x-\l z)+\phi(x)\\
       &\geq -\sigma(-\l_n z+\l z) -\sigma(x-\l z)+\phi(x)\\
       &=-(\l_n-\l)\sigma(-z)-\sigma(x-\l z)+\phi(x).
   \end{align*}
   By letting $n\to \infty$, we have
  $\Phi_\rho(x-\l z) =\phi(x)- \l\rho \le \sigma(x-\l z)$ for all $x \in Y$. For an arbitrary $\l \in \L$ and $x\in Y$, we have 
  $\pi_{\l^+}x+\l^+z$ is separated from $\pi_{\l^+}^cx-\l^-z$ (using $\pi_{\l^+}\in \idempotents$). By \Cref{l:disjoint-implies-additive} and our previous findings, we have
    \begin{align*}
        \Phi_\rho(x+\l z) 
        &= \Phi_\rho(\pi_{\l^+}x+\l^+z+\pi_{\l^+}^c x-\l^-z)\\
        &=\Phi_\rho(\pi_{\l^+}x+\l^+z)+\Phi_\rho(\pi_{\l^+}^cx-\l^-z)\\
        &\le \sigma(\pi_{\l^+}x+\l^+z) + \sigma(\pi_{\l^+}^cx-\l^-z)\\
        &=\sigma(\pi_{\l^+}x+\l^+z+\pi_{\l^+}^cx-\l^-z)\\
        &=\sigma(x+\l z).    
    \end{align*}
   Thus we have shown that we can extend $\phi$ to a larger subspace, while the extension remains bounded by $\sigma$. As in the classical case, the result follows by a standard application of Zorn's lemma. 
\end{proof}


As in the classical case, we can extend the Hahn-Banach theorem to the complex case. 

\begin{theorem}[Hahn-Banach for complex $\L$]\label{t:complex HB}
Let $X$ be a complex $\L$-module and let $Y \subsetneq X$ be a subspace. Furthermore, let $\phi \colon Y \to \L$ be $\L$-linear and $\sigma \colon X \to \L^+$ be a seminorm such that $|\phi(x)| \le \sigma(x)$ for all $x \in Y$. Then there is an $\L$-linear extension $\Phi$ of $\phi$ to all of $X$ such that $|\Phi(x)| \le \sigma(x)$ for all $x \in X$.    
\end{theorem}
\begin{proof}
We can also view $X$ as an $\L_\R$-module and $\Re(\phi)$ satisfies $\Re(\phi)(x) \le |\phi(x)| \le \sigma(x)$ for all $x \in Y$, so by \Cref{t:HB}, $\Re(\phi)$ has a $\L_\R$-linear extension $\Psi$ to all of $X$ such that $\Psi(x) \le \sigma(x)$ for all $x \in X$. Now $\Phi \colon X \to \L$ defined by $\Phi(x) := \Psi(x)-i\Psi(ix)$ is $\L$-linear and extends $\phi$, since $\Re(\Phi) = \Psi$. Now, for every $x \in X$, denote by $\gamma_x \in \L$ the element defined in \Cref{c: Rotation by scalar multiplication} such that $\gamma_x \Phi(x) = |\Phi(x)|$ and $|\gamma_x| \leq 1$. Then $\Phi(\gamma_x x) = \gamma_x \Phi(x) = |\Phi(x)| \in \L_\R$ and so 
\[
        |\Phi(x)| = \Phi(\gamma_x x) =  \Psi(\gamma_x x) \leq \sigma(\gamma_x x) = |\gamma_x| \sigma(x) \leq \sigma(x). \hfill \qedhere
\]
\end{proof}

The next corollary shows that any $\L$-normed space $X$ has a norming dual.

\begin{corol}\label{c:norming functional}
    Let $X$ be an $\L$-normed space. For every $x \in X$ there is an $x^* \in X^*$ with $\|x^*\|=\pi_x$ such that $x^*(x)=\|x\|$.  
\end{corol}

\begin{proof}
    Let $x \in X$ and define the $\L$-linear map $\phi\colon \L x\to \L$ by $\phi(\l x):=\l \|x\|$. Clearly $\phi(x) = \norm{x}$, and for $\l \in \L$
    \[
    |\phi(\l x)| = |\l| \|x\| = |\l| \pi_{\norm{x}} \norm{x} = \pi_{\norm{x}} \norm{\l x},
    \]
    and so $|\phi|$ is dominated by the seminorm $y \mapsto \pi_{\norm{x}} \norm{y}$ thus by Hahn-Banach $\phi$ extends to an $\L$-linear map $x^* \colon X \to \L$ satisfying $|x^*(y)| \leq \pi_{\norm{x}} \norm{y}$ for all $y \in X$. Therefore $x^* \in X^*$ with $\norm{x^*} \leq \pi_{\norm{x}}$.

    On the other hand, using \Cref{l:range_projection} we obtain
    \[ 
    \phi\left( \frac{x}{\|x\|+n^{-1}} \right) = \frac{\norm{x}}{\norm{x} + n^{-1}} \uparrow \pi_{\norm{x}}
    \]
    which implies that $\norm{x^*} \geq \norm{\phi} \geq \pi_{\norm{x}}$. \Cref{l:support_equals_support_of_norm} yields $\norm{x^*} = \pi_{\norm{x}} = \pi_x$ which proves the corollary. 
\end{proof}

\begin{corol}
    For any $\L$-normed space, the map $J \colon X \to X^{**}$ defined by $J(x)(x^*):=x^*(x)$ defines an $\L$-linear isometry. 
\end{corol}

\begin{proof}
   We will only prove the fact that $J$ is an isometry as the $\L$-linearity of $J$ is clear. Let $x \in X$. Then for any $x^* \in B_{X^*}$ it follows that $|J(x)(x^*)|= |x^*(x)|\le \|x^*\|\|x\|\le \|x\|,$ hence $\|J(x)\|\le \|x\|$. By \Cref{c:norming functional} for any $x \in X$ there is an $x^* \in B_{X^*}$ such that $x^*(x)=\|x\|$. This shows that $\|J(x)\| \ge \|x\|$ for all $x \in X$, so $J$ is an isometry.
\end{proof}

We can use this embedding and the fact that the bidual is complete to show the existence of a completion of an $\L$-normed space.

\begin{theorem}
Let $X$ be an $\L$-normed space. Then there exists an $\L$-Banach space $\hat{X}$ such that $X$ is linearly isometric to a dense subspace of $\hat{X}$.
\end{theorem}
\begin{proof}
Let $\hat{X} = \overline{J(X)}$. Then $\hat{X}$ is complete as it is a closed subspace of $X^{**}$, which is complete. By \Cref{c:closure_equals_adherence}, $J(X)$ is dense in $\overline{J(X)}$ and we know that $J$ is an $\L$-linear isometry. 
\end{proof}

For an $\L$-normed space $X$, we define the completion $\hat{X}$ abstractly as an $\L$-Banach space such that for every $\L$-Banach space $Y$, every $T \in B(X,Y)$ extends uniquely to $\hat{T} \in B(\hat{X}, Y)$ with the same norm. It then follows, just as in the classical case, that the completion is unique up to a unique isometric isomorphism.

\section{$\L$-inner product spaces}\label{s:hilbert}

We now consider inner product spaces, and we will consider the real and the complex cases simultaneously; if $\L$ is real, then $\overline\l=\l$.

\begin{defn}
Let $X$ be an $\L$-vector space. An {\it $\L$-semi inner product} 
on $X$ is a map $\ip{\cdot}{\cdot} \colon X \times X \to \L$ satisfying the following properties for all $x,y,z\in X$ and $\lambda \in \L$
\begin{enumerate}
\item $\ip{x}{x} \geq 0$ for all $x \in X$, i.e., positive semidefinite;
\item $\ip{x}{y} = \overline{\ip{y}{x}}$, i.e., (conjugate) symmetric;
\item $\ip{x+y}{z}=\ip{x}{z}+\ip{y}{z}$; and 
\item $\ip{\lambda x}{y}=\lambda \ip{x}{y}$.
\end{enumerate}
Furthermore, if $\ip{x}{x}=0$ implies $x=0$, the $\ip{\cdot}{\cdot}$ is called an $\L$-inner product and $X$ is called an $\L$-inner product space.  
\end{defn}

\begin{remark}
As in the classical case, if $X$ is an $\L$-inner product space the mapping 
\[X\ni x\mapsto \sqrt{\ip{x}{x}}\]
is an $\L$-norm on $X$. If, with this norm, $X$ is an $\L$-Banach space, then we call $X$ an $\L$-Hilbert space.
\end{remark}

As expected, for any non-empty set $S$, the space $\ell^2(S)$ is an $\L$-Hilbert space when equipped with the standard inner product $\ip{\cdot}{\cdot} \colon \ell^2(S) \times \ell^2(S) \to \L$ defined by $\ip{f}{g} \defeq \sum_{s\in S}f(s)\overline{g(s)}$, which induces the norm previously defined on $\ell^2(S)$.


The next theorem follows from representation theory and the classical Cauchy-Schwarz inequality. 

\begin{theorem}[Cauchy-Schwarz inequality] \label{thm:cauchy-schwarz}
Let $X$ be an $\L$-semi inner product space. Then, for any $x,y\in X$, we have
$|\langle x, y\rangle|\leq\norm{x}\norm{y}.$
\end{theorem}

Following the proof of \cite[Theorem~3.1]{BusSch3} verbatim, except identifying the classical field scalars $\lambda$ in \cite[Theorem~3.1]{BusSch3} with $\lambda\cdot 1\in\L$ here, we actually obtain 
 a ``Cauchy-Schwarz equality" for $\L$-semi inner products on $\L$-vector spaces.

\begin{theorem}\label{T: CSI}\cite[Theorem~3.1]{BusSch3} If $X$ is a real, respectively complex, $\L$-semi inner product space, and $\mathbb{K}$ denotes $\R$, respectively $\C$, then for all $x,y\in X$, we have
\[
|\langle x, y\rangle| = \|x\|\|y\| - \frac{1}{2}\underset{\xi\in\mathbb{K}\setminus\{0\}}{\inf}\langle \xi x - |\xi|^{-1}y,\  \xi x - |\xi|^{-1}y\rangle.
\]
\end{theorem}

By following the classical proofs, which use only algebra, we obtain the Pythagorean theorem and the parallelogram law for $\L$-semi-inner product spaces.

\begin{theorem}[Pythagoras] \label{thm:pythagoras}
Let $H$ be an $\L$-inner product space. Then, for any $x,y\in H$, we have
\[\text{ if }\ip{x}{y}=0, \text{ then } \norm{x+y}^2=\norm{x}^2+\norm{y}^2.\]
\end{theorem}
\noindent Note that the converse of Theorem \ref{thm:pythagoras} is true if $\L$ is real.

\begin{theorem}[Parallelogram law]
\label{thm:PL}
Let $H$ be an $\L$-inner product space. Then, for any $x,y\in H$, we have
\begin{equation}\label{e:PL}
\norm{x+y}^2+\norm{x-y}^2=2\norm{x}^2+2\norm{y}^2.
\end{equation}
\end{theorem}

We now investigate the converse of Theorem \ref{thm:PL}. 

\begin{theorem}\label{t:plaw-converse}
    Let  $(H,\norm{\cdot})$ be an $\L$-normed space, with the property that for any $x,y\in H$,\[\norm{x+y}^2+\norm{x-y}^2=2\norm{x}^2+2\norm{y}^2.\] 
    Then, there exists an $\L$-inner product $\ip{\cdot}{\cdot}$ on $H \times H$
    such that $\ip{\cdot}{\cdot}$ induces the $\L$-norm $\norm{\cdot}$ on $H$.
\end{theorem}

\begin{proof}
     Let $R \colon H\times H \to \L$ be defined by
     \[
     R(x,y)\defeq\frac14(\norm{x+y}^2-\norm{x-y}^2).
     \]
     If $\L$ is complex, we also define $\ip{\cdot}{\cdot}\colon H\times H \to \L$ by
     \begin{align}
     \label{e:polarisation} \ip{x}{y} &\defeq R(x,y)-iR(ix,y)\\
   \nonumber &=\frac14(\norm{x+y}^2-\norm{x-y}^2-i\norm{ix+y}^2+i\norm{ix-y}^2)\\
   \nonumber&=\frac14(\norm{x+y}^2-\norm{x-y}^2+i\norm{x+iy}^2-i\norm{x-iy}^2),
     \end{align}
    where we use the identity $\norm{z}=\norm{-iz}$ in the last step above. Note that both $R$ and $\ip{\cdot}{\cdot}$ are continuous in each argument. The proofs for additivity of the first argument and (conjugate) symmetry of $R$ (and $\ip{\cdot}{\cdot}$) follow exactly those of the classical case \cite[Theorem 1]{jordanvneumann}.
    To prove homogeneity of the first argument we first state and prove some claims.
    
    \noindent {\bf Claim 1.} For all $x,y\in H$ and $\pi\in \idempotents$, we have
    \[R(\pi x,y)=\pi R(x,y).\]
    \noindent {\it Proof of Claim 1.} Let $x,y\in H$ and $\pi\in \idempotents$. Using \Cref{l:disjoint-implies-additive}(ii), from the third to the fourth equality, we have
   \begin{align*}
      R(\pi x,y)&=\frac14\left[\norm{\pi x+y}^2-\norm{\pi x-y}^2\right]\\&=\frac14\left[\norm{\pi x+\pi y+\pi^c y}^2-\norm{\pi x-\pi y-\pi^c y}^2\right]\\&=\frac14\left[\norm{\pi(x+y)+\pi^c y}^2-\norm{\pi(x-y)-\pi^c y}^2\right]\\&=\frac14\left[(\pi\norm{x+y}+\pi^c\norm{ y})^2-(\pi\norm{x-y}+\pi^c \norm{y})^2\right]\\&=\frac14\left[\pi\norm{x+y}^2+\pi^c\norm{ y}^2-\pi\norm{x-y}^2-\pi^c \norm{y}^2\right]\\
       &=\frac{\pi}{4}\left[\norm{x+y}^2-\norm{x-y}^2\right]=\pi R(x,y).
   \end{align*}
    
    \noindent {\bf Claim 2.} For all $x,y\in H$ and $t\in \R$, we have \[R(tx,y)=tR(x,y).\]
    \noindent {\it Proof of Claim 2.}
    Following the idea in the classical case, we use the additivity of the first argument and induction to show that
   \[R(nx,y)=nR(x,y)\]
   for all $n\in \mathbb{N}$. 
   Then we extend this to $t\in \mathbb{Z}$, then to $t\in \mathbb{Q}$, and finally $t\in \mathbb{R}$ which follows from the continuity of $R$. 
   
    By Claims 1 and 2, we conclude that for all $x,y\in H$,  idempotents $\pi_1,\dots, \pi_n\in \L_\R$, and $t_1,\dots, t_n\in \R$, we have
    \begin{equation}\label{e:linear-combination}
   R((t_1 \pi_1+\dots+t_n \pi_n)x,y)=(t_1\pi_1+\dots+t_n\pi_n)R(x,y).
    \end{equation}
    Let $\lambda\in \L_\R$. By Freudenthal's Spectral Theorem (\Cref{t:freudenthal}), there is a sequence $(\lambda_n)$ of step functions such that $\lambda_n\to \lambda$. Therefore,    
    \[R(\l x,y)=\lim_{n\to\infty}R(\lambda_nx,y)=\lim_{n\to\infty} \lambda_nR(x,y)=\lambda R(x,y),\]
  because $R$ is continuous.
    Combining the above with \eqref{e:polarisation}, we have 
    \begin{align}
    \nonumber \ip{\l x}{y}&=R(\l x,y)-iR(i\l x,y)\\
    \label{e:homogeneity-real}    &=\l R(x,y)-i\l R(i x,y)=\l\ip{x}{y}.
     \end{align}
     Note that 
    \begin{align}
    \nonumber \ip{ix}{y}&=\frac14(\norm{ix+y}^2-\norm{ix-y}^2+i\norm{ix+iy}^2-i\norm{ix-iy}^2)\\
    \nonumber &=\frac14(\norm{x-iy}^2-\norm{x+iy}^2+i\norm{x+y}^2-i\norm{x-y}^2)\\
    \nonumber &=\frac i4(-i\norm{x-iy}^2+i\norm{x+iy}^2+\norm{x+y}^2-\norm{x-y}^2)\\
    \label{eq:homogeneity-i} &=i\ip{x}{y}.
    \end{align}
    Finally, using \eqref{e:homogeneity-real}, \eqref{eq:homogeneity-i}, and the additivity of the first argument, we conclude that $\ip{\l x}{y}=\l \ip{x}{y}$ for any $\l \in \L$.
   It is straightforward to verify that $\norm{x}=\sqrt{R(x,x)}$ and $\norm{x}=\sqrt{\ip{x}{x}}$ for any $x\in H$, and this completes the proof. 
\end{proof}

Consider a set $S$ containing at least two distinct elements $s,t\in S$. For $1 \leq p \leq \infty$, define the elements $f,g\in \ell^p(S)$ by $f \defeq e_s +  e_t$ and $g\defeq  e_s -  e_t$. It is then straightforward to check that the parallelogram law $\norm{f+g}_p^2+\norm{f-g}_p^2=2\norm{f}_p^2+2\norm{g}_p^2$ holds if and only if $p=2$. 

\subsection{Projection property}

For this section, we follow the structure of \cite[Chapter 2]{conway}. Recall the definition of  $\L$-convex sets in an $\L$-vector space given in Definition \ref{defn:l-convex}.

\begin{theorem}\label{t:projection property}
Let $H$ be an $\L$-Hilbert space and $K$ be a closed, $\L$-convex nonempty subset of $H$. Then, for any $x\in H$, there exists a unique point $k_0$ in $K$ such that
\[\norm{x-k_0}=\inf_{k\in K}\norm{x-k}.\]
\end{theorem}

\begin{proof}
Without loss of generality, we assume $x=0$, since translation preserves $\L$-convexity and is an isometry.
Let $\lambda = \inf_{k \in K} \norm{k}$. By \Cref{l:zero_distance_convex}, there is a net $(k_\alpha)_{\alpha \in A}$ in $K$ such that $\norm{k_\alpha} \downarrow \lambda$. Define
\[ \Eps \defeq \left\{ \sqrt{\norm{k_\alpha}^2 - \lambda^2} \colon \alpha \in A \right\}, \]
then $\Eps \searrow 0$ since $\norm{k_\alpha} \downarrow \lambda$. Let $\eps \in \Eps$, then there is an $\alpha_0$ with $\eps^2 = \norm{k_{\alpha_0}}^2 - \lambda^2$. Let $\alpha, \beta \geq \alpha_0$. Then $\lambda \leq \norm{k_\alpha}, \norm{k_\beta} \leq \norm{k_{\alpha_0}}$, and note that $(k_\alpha+k_\beta)/2\in K$ since $K$ is $\L$-convex, and so by the parallelogram law,
\begin{align*}
\norm{\frac12(k_\alpha-k_\beta)}^2&=\frac12 \norm{k_\alpha}^2+\frac12 \norm{k_\beta}^2-\norm{\frac12(k_\alpha+k_\beta)}^2\\
&\leq \frac12 \norm{k_{\alpha_0}}^2+\frac12 \norm{k_{\alpha_0}}^2-\l^2=\eps^2,
\end{align*}
so $(k_\alpha)$ is Cauchy and hence $k_\alpha \to k_0$ for some $k_0 \in K$. By continuity of the norm, $\lambda = \lim_{\alpha \in A} \norm{k_\alpha} = \norm{k_0}.$
    
To show uniqueness, suppose there exists $k_1\in K$ such that $\norm{k_1}=\l$, then
\[\l\leq \norm{\frac12(k_0+k_1)}\leq \frac12\norm{k_0}+\frac12\norm{k_1}=\l,\]
and so $\l=\norm{\frac12(k_0+k_1)}$. Furthermore, by the parallelogram law, 
\begin{align*}
    \l^2&=\norm{\frac12(k_0+k_1)}^2\\ &=\frac12\norm{k_0}^2+\frac12\norm{k_1}^2-\norm{\frac12(k_0-k_1)}^2\\
    &=\l^2-\norm{\frac12(k_0-k_1)}^2,
\end{align*}
which implies that $\norm{k_0-k_1}=0$ and thus $k_0=k_1$.
\end{proof}

At first glance it may seem as if the full assumption of $\L$-convexity is not used in the proof of \Cref{t:projection property}, since it seems that only $\idempotents$-convexity (in the application of \Cref{l:zero_distance_convex}) and the fact that if $x,y \in K$, then $(x+y)/2 \in K$ are used. However, note that these two properties yield that $K$ is convex with respect to any step function with dyadic rational coefficients, and the closedness of $K$ together with the Freudenthal Spectral theorem (\Cref{t:freudenthal}) yield $\L$-convexity.

\begin{theorem}\label{t:distance-projection}
Let $H$ be an $\L$-Hilbert space and $M$ be a closed subspace of $H$. Let $x\in H$. If $y$ is the unique element of $M$ such that
$\norm{x-y}=\inf_{z\in M}\norm{x-z}$, then $\ip{x-y}{z}=0$ for all $z\in M$. Conversely, if $y\in M$ such that $\ip{x-y}{z}=0$ for all $z\in M$, then $\norm{x-y}=\inf_{z\in M}\norm{x-z}$. 
\end{theorem}

\begin{proof}
    Assume that $y$ is the unique element of $M$ such that $\norm{x-y}=\inf_{z\in M}\norm{x-z}.$ Let $z\in M$. Then $y+z\in M$ and so 
    \begin{align*} 
    \norm{x-y}^2&\leq  \norm{x-(y+z)}^2\\
    &= \ip{(x-y)-z}{(x-y)-z} \\
    &=\norm{x-y}^2-2\Re\ip{x-y}{z}+\norm{z}^2,
    \end{align*}
    which implies that \begin{equation}\label{eq:orthogonal-1}
    2\Re\ip{x-y}{z}\leq \norm{z}^2.
    \end{equation}
    Therefore, for any $\zeta\in \mathbb{T}$, we have $2\Re\ip{x-y}{\overline{\zeta}z}\leq \norm{z}^2,$
    or equivalently, 
    \begin{equation}\label{eq:orthogonal-2}
    2\Re(\zeta\ip{x-y}{z})\leq \norm{z}^2.
    \end{equation}
    From \eqref{eq:orthogonal-1}, \eqref{eq:orthogonal-2}, and \eqref{e:defn-modulus} in \Cref{r:real or complex}, we conclude that
    $2\left|\ip{x-y}{z}\right|\leq \norm{z}^2.$
    Now we replace $z$ with $tz$ for arbitrary $0<t\in \R$, we have
     $2\left|\ip{x-y}{z}\right|\leq t\norm{z}^2,$
     and taking $t\to 0$, we have $\ip{x-y}{z}=0$.

     Conversely, assume $y\in M$ is such that $\ip{x-y}{z}=0$ for all $z\in M$. Let $z\in M$. Then,
     $\ip{x-y}{y-z}=0$, and by Pythagoras, we have 
     \[
         \norm{x-z}^2
        =\norm{x-y+y-z}^2=\norm{x-y}^2+\norm{y-z}^2\geq \norm{x-y}^2,
     \]
     and so $\norm{x-y}\leq \norm{x-z}.$ Taking the infimum over all $z \in M$ concludes the proof.
\end{proof}

Let $H$ be an $\L$-Hilbert space with $A \subseteq H$. As usual, the \emph{orthocomplement} of the set $A$ is defined as 
\[
            A^\perp \defeq \{x\in H: \ip{x}{y}=0 \text{ for all } y\in A\}.
\]We list a few standard results relating to orthocomplements which are needed in the sequel.

\begin{prop}
Let $H$ be an $\L$-Hilbert space with $A \subseteq H$. Then $A^{\perp}$ is a closed subspace of $H$.
\end{prop} 

\begin{prop}\label{p: Subspace and closure have the same orthocomplement}
 Let $H$ be an $\L$-Hilbert space and let $M$ be a subspace of $H$. Then $M^\perp = (\overline{M})^\perp$.   
\end{prop}

\begin{proof}
Since $M \subseteq \overline{M}$, we obtain the inclusion $(\overline{M})^\perp \subseteq M^\perp$ since orthocomplementation is subset reversing. For the reverse inclusion, take $y \in M^\perp$ and consider $x\in \overline{M}$. Since there exists a net $(x_\alpha)$ in $M$ such that $x_\alpha \to x$ and $\ip{y}{x_\alpha} = 0$ for all $\alpha$, we conclude by the continuity of the inner product that $y \in (\overline{M})^\perp$ and thus $M^\perp \subseteq (\overline{M})^\perp$.
\end{proof}

\begin{corol}\label{c:decomposition}
Let $H$ be an $\L$-Hilbert space and let $M$ be a closed subspace of $H$. Then $H=M \oplus M^\perp$.
\end{corol}

\begin{proof}
Let $x\in H$. By \Cref{t:distance-projection}, there exists a unique element $y\in M$ such that $\ip{x-y}{z}=0$ for all $z\in M$, that is $x-y\in M^\perp$, and this completes the proof.
\end{proof}



If $M$ is a closed subspace of $H$ and $x\in H$, then there is a unique $y\in M$ such that $x-y\in M^\perp$. Define $P_M\colon H\to M$  by $P_Mx \defeq y$. Note that $P_M$ is well defined, by the uniqueness of the corresponding element in $M$ given by \Cref{t:distance-projection}.

\begin{prop}
The function $P_M$ satisfies the following properties:
\begin{enumerate}[(i)]
\item $P_M$ is $\L$-linear and contractive;
\item $P_M^2=P_M$;
\item $\mathrm{ker}(P_M)=M^\perp$; and
\item $\mathrm{ran}(P_M)=M$.
\item $P_{M^\perp} = I - P_M$ where $I:H\to H$ is the identity function on $H$.
\end{enumerate}
\end{prop}

\begin{proof}
The proof of (i) - (iv) follows along the same line as in the classical case which can be found in \cite[Theorem~I.2.7]{conway}. The statement in (v) follows directly from the observation that $I = P_M + P_{M^\perp}$.
\end{proof}


The following result follows as in the classical case \cite[Corollary~I.2.9]{conway}. 

\begin{prop}\label{p: Double orthocomplement of closed subspace}
Let $H$ be an $\L$-Hilbert space and let $M$ be a closed subspace of $H$. Then $(M^\perp)^{\perp} = M$. 
\end{prop}

The notation $\overline{\Span}\left( A \right)$ denotes the closure of the subspace generated by $A$. 

\begin{prop}\label{p: Closed linear span is the double orthocomplement}
Let $H$ be an $\L$-Hilbert space with $A \subseteq H$. Then $(A^\perp)^{\perp} = \overline{\Span}\left( A \right)$.     
\end{prop}

\begin{proof}
Since $A \subseteq (A^\perp)^{\perp}$ and $(A^\perp)^{\perp}$ is a submodule of $H$, it is clear that $\Span \left( A \right) \subseteq (A^\perp)^{\perp}$ and since $(A^\perp)^{\perp}$ is closed, we have $\overline{\Span}\left( A \right) \subseteq (A^\perp)^{\perp}$. For the reverse inclusion, since $A \subseteq \overline{\Span}\left( A \right)$ and orthocomplementation is subset reversing, we conclude by \Cref{p: Double orthocomplement of closed subspace} that $(A^\perp)^{\perp} \subseteq ( \overline{\Span}\left( A \right)^\perp )^\perp = \overline{\Span}\left( A \right)$.
\end{proof}

\begin{corol}\label{c: Characterisation of density}
Let $H$ be an $\L$-Hilbert space and let $M$ be a subspace of $H$. Then $H = \overline{M}$ if and only if $M^\perp = \{ \0 \}$.    
\end{corol}

\begin{proof}
First assume that $H = \overline{M}$, then by \Cref{p: Subspace and closure have the same orthocomplement} we have $M^\perp = (\overline{M})^\perp = H^\perp = \{ \0 \}$. Conversely, if $M^\perp = \{ \0 \}$, by \Cref{p: Closed linear span is the double orthocomplement} we have $\overline{M} = \overline{\Span}\left( M \right) = (M^\perp)^{\perp} = \{ \0 \}^\perp = H$.
\end{proof}

\subsection{Suborthonormal systems}

We follow the terminology of \cite[Section~2.3]{edeko2023decomposition}.

\begin{defn}
Let $H$ be an $\L$-inner product space. A collection $(x_i)_{i\in I}$ in $H$ is called an \emph{orthogonal system} if $\ip{x_i}{x_j} = \mathbf{0}$ whenever $i \neq j$. An orthogonal system $(x_i)_{i\in I}$ is \emph{suborthonormal} if each $x_i$ is normalised (i.e. $\norm{x_i} \in \idempotents$ for every $i\in I$). Furthermore, if $\pi \in \idempotents$, a suborthonormal system $(x_i)_{i\in I}$ is called \emph{$\pi$-homogeneous} if $\ip{x_i}{x_j} = \pi \delta_{ij}$, and it is \emph{homogeneous} if it is $\pi$-homogeneous for some $\pi \in \idempotents$. A homogeneous suborthonormal system $(x_i)_{i\in I}$ is \emph{orthonormal} if it is $1$-homogeneous. A (sub)orthonormal system $\mathcal{B} \defeq (x_i)_{i\in I}$ in $H$ is called a \emph{(sub)orthonormal  basis} if $\mathcal{B}^{\perp} = \{ 0 \}$.
\end{defn}


We now formulate the Gram-Schmidt Orthogonalisation process in the setting of $\L$-Hilbert spaces. We omit the proof since it follows exactly like in the classical case. In the classical case, the result can be formulated in any inner product space. However, in our case we need to consider an $\L$-Hilbert space since the existence of normalisations of vectors cannot be guaranteed without completeness by \Cref{l: Normalisation exists in an L-Banach space}.

\begin{prop}[Gram-Schmidt Orthogonalisation process]
Let $H$ be an $\L$-Hilbert space with $\mathcal{L} \defeq (x_i)^n_{i = 1}$ a collection of linearly independent vectors. Then there exists a suborthonormal system $\mathcal{S} \defeq (e_i)^n_{i = 1}$ such that $\Span(\mathcal{L})= \Span(\mathcal{S})$. 
\end{prop}

The following result is proven in \cite[Proposition~2.11]{edeko2023decomposition}. We omit the proof here since it follows via a routine application of Zorn's lemma.

\begin{prop}\label{p : suborthonormal basis using zorn}
Let $H$ be an $\L$-Hilbert space. If $\mathcal{S} \subseteq H$ is a suborthonormal system, there exists a suborthonormal basis $\mathcal{B} \subseteq H$ such that $\mathcal{S} \subseteq \mathcal{B}$. 
\end{prop}

\begin{corol}
 Every $\L$-Hilbert space possesses a suborthonormal basis.      
\end{corol}

\begin{theorem}[Bessel's inequality]
Let $H$ be an $\L$-Hilbert space with $\mathcal{S}$ a suborthonormal system. For every $x\in H$ we have
\[
        \sum_{e\in \mathcal{S}} \left| \ip{x}{e} \right|^2 \leq \norm{x}^2.  
\]
\end{theorem}

\begin{proof}
Consider a finite set $F \in \mathcal{F}(\mathcal{S})$.  We define
\[
        x_F \defeq x - \sum_{e\in F} \ip{x}{e} e.
\] 
Let $f\in F$ be arbitrary. Then
\[
        \ip{x_F}{f} = \ip{x}{f} - \sum_{e\in F} \ip{x}{e} \ip{e}{f} = \ip{x}{f} - \ip{x}{f} \norm{f}, 
\]
and by \Cref{p: normalise}~(iv) we have
\begin{align}\label{eq: norm slips in}
        \ip{x}{f} \norm{f} = \ip{x}{f} \overline{\norm{f}} = \ip{x}{\norm{f} f} = \ip{x}{f}.
\end{align}
Thus $\ip{x_F}{f} = 0$ and then it is clear that $\ip{x_F}{\sum_{e\in F} \ip{x}{e} e} = 0$. By \Cref{thm:pythagoras}, we have
\[
        \norm{x}^2 = \norm{x_F + \sum_{e\in F} \ip{x}{e} e}^2 = \norm{x_F}^2 + \norm{\sum_{e\in F} \ip{x}{e} e}^2 \geq \norm{\sum_{e\in F} \ip{x}{e} e}^2.  
\]
Since $\mathcal{S}$ is a suborthonormal system, it follows by another application of \Cref{thm:pythagoras} and the argument in \eqref{eq: norm slips in} that
\[
        \norm{\sum_{e\in F} \ip{x}{e} e}^2 = \sum_{e\in F} \left| \ip{x}{e} \right|^2 \norm{e}^2 = \sum_{e\in F} \left| \ip{x}{e} \right|^2.
\] Thus $\sum_{e\in F} \left| \ip{x}{e} \right|^2 \leq \norm{x}^2$ for every $F \in \mathcal{F}(\mathcal{S})$. Since
\[
        \sum_{e\in \mathcal{S}} \left| \ip{x}{e} \right|^2 = \sup\left\lbrace \sum_{e\in F} \left| \ip{x}{e} \right|^2 ~:~ F \in \mathcal{F}(\mathcal{S})\right\rbrace
\] we conclude that $\sum_{e\in \mathcal{S}} \left| \ip{x}{e} \right|^2 \leq \norm{x}^2$. 
\end{proof}

\begin{theorem}
Let $H$ be an $\L$-Hilbert space with $\mathcal{S}$ a suborthonormal system. For every $x\in H$, the formal series $\sum_{e\in \mathcal{S}} \ip{x}{e}e$ is convergent in $H$. 
\end{theorem}


\begin{proof}
It suffices to verify that the net $\left( \sum_{e\in F} \ip{x}{e}e \right)_{F\in \mathcal{F}(\mathcal{S)}}$ is Cauchy. Since $\sum_{e\in \mathcal{S}} \left| \ip{x}{e} \right|^2$ is convergent, it is Cauchy, so there exists a set $\Eps \searrow 0$ such that for every $\eps \in \Eps$ there exists an $F_0 \in \mathcal{F}(\mathcal{S)}$ such that for all finite $F ,G \supseteq F_0$
\[
        \left| \sum_{e \in F} \left| \ip{x}{e} \right|^2 - \sum_{e \in G} \left| \ip{x}{e} \right|^2\right| \leq \eps.
\]Fix $\eps \in \Eps$ and let $F_0 \in \mathcal{F}(\mathcal{S)}$ be as above. For all finite $F,G \supseteq F_0$, we have 
\begin{align*}
        &\norm{\sum_{e\in F} \ip{x}{e}e - \sum_{e\in G} \ip{x}{e}e}^2\\
        =\ &\norm{\sum_{e\in F\setminus G} \ip{x}{e}e - \sum_{e\in G\setminus F} \ip{x}{e}e}^2\\
        =\ & \norm{\sum_{e\in F\setminus G} \ip{x}{e}e}^2 + \norm{\sum_{e\in G\setminus F} \ip{x}{e}e}^2\\
        =\ &\sum_{e\in F\setminus G} \left| \ip{x}{e} \right|^2 + \sum_{e\in G\setminus F} \left| \ip{x}{e} \right|^2 \\
        = &\sum_{e\in F} \left| \ip{x}{e} \right|^2 - \sum_{e\in F \cap G} \left| \ip{x}{e} \right|^2 + 
        \sum_{e\in G} \left| \ip{x}{e} \right|^2 - \sum_{e\in F \cap G} \left| \ip{x}{e} \right|^2 \leq 2\eps. \qedhere
\end{align*}
\end{proof}

\begin{prop}\label{p:suborthonormal basis equivalences}
Let $H$ be an $\L$-Hilbert space with $\mathcal{S}$ a suborthonormal system. Then the following statements are equivalent:
\begin{itemize}
    \item[(i)] $\mathcal{S}$ is a maximal suborthonormal system.
    \item[(ii)] $\mathcal{S}$ is a suborthonormal basis.
    \item[(iii)] $\overline{\Span}(\mathcal{S}) = H$. 
    \item[(iv)] For every $x\in H$, $x = \sum_{e\in \mathcal{S}} \ip{x}{e}e$.
    \item[(v)] For every $x,y \in H$, $\ip{x}{y} = \sum_{e\in \mathcal{S}} \ip{x}{e} \ip{e}{y}$. 
    \item[(vi)] For every $x\in H$,  $\norm{x}^2 = \sum_{e\in \mathcal{S}} \left| \ip{x}{e} \right|^2$ \quad (Parseval's identity).
\end{itemize}
\end{prop}

\begin{proof}
The proof of these equivalences is similarly to the classical case which can be found in \cite[Theorem~I.4.13]{conway}. We note in particular that the equivalence (ii) $\iff$ (iii) is proven in \Cref{c: Characterisation of density}.
\end{proof}

\subsection{Representing $\L$-Hilbert spaces as $\ell^2$-spaces}\label{ss:representation theory for L-Hilbert spaces}

In this section we prove that any $\L$-Hilbert space $H$ can be decomposed as a direct sum of $\ell^2$-spaces. We first show that a Hilbert space with a $\pi$-homogeneous suborthonormal basis is isomorphic to $\ell^2$.

\begin{theorem}\label{t:pi homogeneous hilbert space}
    Let $\pi \in \idempotents$ and let $H$ be an $\L$-Hilbert space with a $\pi$-homogeneous suborthonormal basis $\mathcal{B}$. Then $H \cong \ell^2(\mathcal{B}, \pi \L)$.
\end{theorem}
\begin{proof}
    Define $T \colon H \to \ell^2(\mathcal{B}, \pi\L)$ and $S \colon \ell^2(\mathcal{B}, \pi \L) \to H$ by
    \[ 
    Tx \defeq \left( \ip{x}{e}\right)_{e \in \mathcal{B}} \quad \quad S\left( (\l_e)_{e \in \mathcal{B}} \right) \defeq \sum_{e \in \mathcal{B}} \l_e e.
    \]
    Let $(\l_e)\in  \ell^2(\mathcal{B}, \pi \L)$. Since $\mathcal{B}$ is a suborthonormal set, we have
    \[\norm{\sum_{e \in \mathcal{B}} \l_e e}^2= \sum_{e \in \mathcal{B}} \norm{\l_e e}^2=\pi \sum_{e \in \mathcal{B}} |\l_es|^2,\]
    and so the series $\sum_{e \in \mathcal{B}} \l_e e$ converges.    
    By (iv) of \Cref{p:suborthonormal basis equivalences},  $T$ is well defined and $ST = I_H$. For any $(\lambda_b)_{b\in\mathcal{B}}\in \ell^2(\mathcal{B}, \pi \L)$, we have
    \[(\lambda_b)_{b\in\mathcal{B}} \stackrel{S}{\longmapsto} \sum_{b\in\mathcal{B}} \lambda_b b\stackrel{T}{\longmapsto}\left(\ip{\sum_{b\in\mathcal{B}} \lambda_b b}{e}\right)_{e\in\mathcal{B}}=\left(\lambda_e\ip{ e}{e}\right)_{e\in\mathcal{B}}=\left(\lambda_e\right)_{e\in\mathcal{B}}\]
    since $\lambda_e\ip{e}{e}=\pi\lambda_e=\lambda_e.$ Therefore $TS = I_{ \ell^2(\mathcal{B},\pi\L)}$  and by Parseval's identity (\Cref{p:suborthonormal basis equivalences} part (vi)), $T$ is an isomorphism of Hilbert spaces.  
    \end{proof}

To show that every $\L$-Hilbert space can be decomposed as a direct sum of homogeneous $\L$-Hilbert spaces, we start by obtaining a suborthonormal basis more refined than the one from \Cref{p : suborthonormal basis using zorn}. 

\begin{prop}\label{p:decreasing suborthonormal basis}
Let $H$ be an $\L$-Hilbert space. Then there exists an ordinal $\gamma$ and a suborthonormal basis $(b_\alpha)_{\alpha \in \gamma}$ of $H$ such that $\alpha \mapsto \norm{b_\alpha}$ is decreasing and $\norm{b_0}=\pi_H$.
\end{prop}
\begin{proof}
    Let $CS(H)$ be the set of closed subspaces of $H$. Using \Cref{l:L-ban-support-is-realized}, let $G \colon CS(H) \to H$ be a choice function such that $\norm{ G(Y)} = \pi_Y$ for all $Y \in CS(H)$. Note that if $X,Y \in CS(H)$ with $X \subseteq Y$, then $\pi_X \leq \pi_Y$, so $Y \mapsto \norm{G(Y)}$ is increasing.

    By transfinite recursion, for an ordinal $\beta$, define $b_\beta \defeq G \left( \{ b_\alpha \colon \alpha < \beta \}^\perp \right)$; note that $b_0 = G(\emptyset^\perp) = G(H)$. By a cardinality argument there is an ordinal $\alpha$ with $b_\alpha = 0$. Let $\gamma$ be the least such ordinal, then $(b_\alpha)_{\alpha \in \gamma}$  satisfies the requirements of the proposition.   
\end{proof}

Note that the next lemma also holds for arbitrary complete Boolean algebras $\idempotents$.

\begin{lemma}\label{l:increasing ordinal function}
    Let $\gamma$ be an ordinal and let $f \colon \gamma + 1 \to \idempotents$ be an increasing function with $f(0) = 0$. For $0 < \beta \leq \gamma$, define $\pi_\beta \defeq f(\beta) - \sup_{a < \beta} f(\alpha)$. Then $\sup_{0 < \beta \leq \gamma} \pi_\beta = f(\gamma)$.
\end{lemma}

\begin{proof}
    Clearly $f(\gamma)$ is an upper bound of $(\pi_\beta)_{0 < \beta \leq \gamma}$. Let $u \in \idempotents$ be an upper bound of $(\pi_\beta)_{0 < \beta \leq \gamma}$. We claim that $u \geq f(\beta)$ for all $0 \leq \beta \leq \gamma$. Note that $u \geq 0 = f(0)$, and for $0 < \beta \leq \gamma$ we prove the claim by transfinite induction, so suppose $u \geq f(\alpha)$ for all $\alpha < \beta$. Since
    \[
    f(\beta) = \pi_\beta + \sup_{\alpha < \beta} f(\alpha) = \pi_\beta \vee \sup_{\alpha < \beta} f(\alpha)
    \]
    and $u \geq \pi_\beta$, it follows that $u \geq f(\beta)$, proving the claim. In particular $u \geq f(\gamma)$ and so $\sup_{0 < \beta \leq \gamma} \pi_\beta = f(\gamma)$.
\end{proof}

If $(H_i)_{i \in I}$ is a collection of $\L$-Hilbert spaces, we define the $\ell^2$-sum $\oplus_{i \in I} H_i$ as in the classical case; note that unlike in the classical case, the elements of $\oplus_{i \in I} H_i$ can have uncountably many nonzero coordinates (similar to the case of \Cref{e:uncountable support}).

\begin{theorem}\label{t:hilbert space disjoint homogeneous sum}
    Let $H$ be an $\L$-Hilbert space. Then there exists a set $I$ and a disjoint collection $(\pi_i)_{i \in I}$ of nonzero idempotents with $\sum_{i \in I} \pi_i = \pi_H$ such that $H = \oplus_{i \in I} \, \pi_i H$ and $\pi_i H$ has $\pi_i$-homogeneous suborthonormal basis for all $i \in I$.    
\end{theorem}

\begin{proof}
    Using \Cref{p:decreasing suborthonormal basis}, we obtain an ordinal $\gamma$ and a suborthonormal basis $(b_\alpha)_{\alpha \in \gamma}$  of $H$ such that $\alpha \mapsto \norm{b_\alpha}$ is decreasing and $\norm{b_0} = \pi_H$. We also define $b_\gamma \defeq \0$. For $0 < \beta \leq \gamma$, define $\pi_\beta$ to be the size of the jump of $\alpha \mapsto \norm{b_\alpha}$ at $\beta$, i.e., $\pi_\beta \defeq \inf_{\alpha < \beta} \norm{b_\alpha} - \norm{b_\beta}$. Define $f \colon \gamma + 1 \to \idempotents$ by $f(\alpha) \defeq \pi_H - \norm{b_\alpha}$, then $f$ is increasing with $f(0) = 0$, and
    \[ 
    \pi_\beta = \inf_{\alpha < \beta} \norm{b_\alpha} - \norm{b_\beta} = \pi_H - \norm{b_\beta} - \sup_{\alpha < \beta} (\pi_H - \norm{b_\alpha}) = f(\beta) - \sup_{\alpha < \beta} f(\alpha).
    \]
    Therefore 
    \begin{equation}\label{e:sup-idempotent}
    \sup_{0 < \beta \leq \gamma} \pi_\beta = f(\gamma) = \pi_H
    \end{equation}
    by \Cref{l:increasing ordinal function}.

    Note that
    \begin{equation}\label{e:pi_beta less than b_alpha}
    \alpha < \beta \Rightarrow \pi_\beta = \inf_{\alpha' < \beta} \norm{b_{\alpha'}} - \norm{b_\beta} \leq \norm{b_\alpha}.
    \end{equation}
    If $0 < \alpha < \beta$, then by definition $\pi_\alpha$ is disjoint with $\norm{b_\alpha}$, which together with $\eqref{e:pi_beta less than b_alpha}$ shows that $\pi_\alpha$ and $\pi_\beta$ are disjoint, hence $(\pi_\beta)_{0 < \beta \leq \gamma}$ is a disjoint collection. Moreover, if $0 < \beta \leq \alpha$, then $\norm{b_\alpha} \leq \norm{b_\beta}$, and $\norm{b_\beta}$ is disjoint from $\pi_\beta$, so $\pi_\beta \norm{b_\alpha} = 0$, and if $ \alpha < \beta$, then \eqref{e:pi_beta less than b_alpha} implies that $\pi_\beta \norm{b_\alpha} = \pi_\beta$. Hence $\norm{\pi_\beta b_\alpha} = \pi_\beta \norm{b_\alpha} \in \{\pi_\beta, 0\}$. Let $I \defeq \{0 < \beta \leq \gamma \colon \pi_\beta \not= 0\}$, then by disjointness of $(\pi_i)_{i\in I}$ and \eqref{e:sup-idempotent}, $\sum_{i \in I} \pi_i = \sup_{0 < \beta \leq \gamma} \pi_\beta = \pi_H$. 
    
    It remains to prove that $\pi_i H$ has a $\pi_i$-homogeneous basis for all $i \in I$. Define 
    \[
    S \defeq \{ (i, \alpha) \in I \times \gamma \colon \pi_i b_\alpha \not= \0 \},
    \]
    and consider the collection $(\pi_i b_\alpha)_{(i, \alpha) \in S}$, which is orthogonal thanks to the disjointness of $(\pi_i)$ and the orthogonality of $(b_\alpha)$, and $\norm{\pi_i b_\alpha} = \pi_i$, so $(\pi_i b_\alpha)_{(i, \alpha) \in S}$ is suborthonormal. If $x \in H$, then the computation
    \[
    \ip{x}{b_\alpha} = \ip{x}{\pi_H b_\alpha} = \ip{x}{\sum_{i \in I} \pi_i b_\alpha} = \sum_{i \in I} \ip{x}{\pi_i b_\alpha}
    \]
    shows that the orthogonal complement of $(\pi_i b_\alpha)_{(i, \alpha) \in S}$ is trivial and hence $(\pi_i b_\alpha)_{(i, \alpha) \in S}$ is a suborthonormal basis. Fix $i \in I$. Define 
    \[S_i\defeq (\{i\}\times \gamma) \cap S. \]
    The disjointness of $(\pi_j)_{j \in I}$ now implies that the closed subspace $\pi_i H$ equals the closed $\L$-linear span of $(\pi_i b_\alpha)_{(i,\alpha) \in S_i}$, so that $(\pi_i b_\alpha)_{(i,\alpha) \in S_i}$ is a suborthonormal basis of $\pi_i H$ with $\norm{\pi_i b_\alpha} = \pi_i$. 
\end{proof}

The above theorem for KH-modules is \cite[Theorem~7.4.7(2)]{domops}, but some steps are missing in that proof. Our proof is very different: we have a simpler argument exploiting transfinite induction and recursion. 

Combining  \Cref{t:hilbert space disjoint homogeneous sum} and \Cref{t:pi homogeneous hilbert space} yields the following corollary.
\begin{corol}\label{c: Representation for L-Hilbert spaces}
    Let $H$ be an $\L$-Hilbert space.  Then there exist a set $I$ and a disjoint collection $(\pi_i)_{i \in I}$ of nonzero idempotents with $\sum_{i \in I} \pi_i = \pi_H$ and a set $S_i$ such that 
    \[H=\bigoplus_{i\in I} \ell^2(S_i,\pi_i\L).\]
\end{corol}

\subsection{Riesz Representation Theorem}

We now move towards proving the Riesz Representation Theorem for $\L$-Hilbert spaces. The statement closely follows the classical Riesz Representation Theorem's statement, but the proof is somewhat more involved. As in the classical case, the Riesz Representation Theorem allows for defining adjoints of operators between $\L$-Hilbert spaces, with which we end the section.


    %
    
    \begin{defn}
        For non-zero $\pi\in \idempotents$, 
        we say that an element $\lambda\in\L$ is \emph{$\pi$-regular}
        if $0\neq \pi\leq \pi_{\lambda}$ and there exists some $\kappa\in\L$ so
        that $\lambda\kappa=\pi.$
    \end{defn}

    \begin{remark}
            The definition of $\pi$-regularity above is important, but may be opaque. The term `regular' is adopted from \cite[Theorem~5]{kaplansky53}, however in this reference the terminology is left undefined, which makes the reference quite difficult to understand.             
            The intuition behind the concept is perhaps easier to understand through invoking representation theory. 
            In $\L$, with $\pi$ an idempotent in $\L$, for an element $\lambda \in \L$ to be $\pi$-regular means that $\lambda$ is bounded away from zero on the support of $\pi$. Therefore $\pi\lambda$ is invertible in $\pi \L$ where the multiplicative unit is, of course, $\pi$.           
    \end{remark}

    Before proving the Riesz Representation Theorem, we prove a number of easy lemmas. 

    Firstly, the product of an element in $\L$ with a non-zero idempotent below the element's support cannot vanish:
    \begin{lemma}   
        \label{l:support-no-kill} Let $X$ be an $\L$-normed space. For all
        non-zero $x\in X$ and all $\pi\in \idempotents$, if $0\neq \pi\leq \pi_{x}$
        then $\pi n_{x}\neq0$ and $\pi x\neq0$.
    \end{lemma}

    \begin{proof}
        Take a non-zero $x\in X$ and $\pi\in \idempotents$ and assume $0\neq \pi\leq \pi_{x}$.
        Then $\|\pi n_{x}\|=\pi\|n_{x}\|=\pi\pi_{x}=\pi\neq0$ implies $\pi n_{x}\neq0$.
        Also, $\|\pi x\|=\|\pi(\|x\|n_{x})\|=\|x\|\|\pi n_{x}\|\neq0$, so that $\pi x\neq0$.
    \end{proof}

    Secondly, any non-zero element of $\L$ has a part that is regular: 
\begin{lemma}\label{l:regularization}
    For non-zero $\lambda\in\L$ there exists a non-zero $\pi \in \idempotents$
    so that $0\neq \pi\leq \pi_{\lambda}$ so that $\lambda$ is $\pi$-regular.
\end{lemma}
\begin{proof}
    By considering the real or imaginary positive or negative part of $\lambda$, we may assume that $\lambda$ is positive.
    By the Freudenthal Spectral Theorem (\Cref{t:freudenthal}), there exist 
    $0<\pi\leq \pi_\lambda \in\idempotents$ and $0<t\in \R$ so that 
    $0<t \pi \leq \lambda$. Now, $\pi$ is the multiplicative unit in $\pi\L$ hence is invertible in $\pi\L$. Furthermore, since $0<t \pi = t\pi^2 = \pi (t\pi) \leq \pi\lambda$ in $\pi\L$, by \Cref{l:invertible}, the element $\pi\lambda$ is invertible in $\pi\L$. Hence there exists $\kappa \in \pi\L \subseteq \L$ 
    so that $\kappa\lambda = \pi \kappa \lambda =  \kappa (\pi \lambda) = \pi$.
\end{proof}

Thirdly, regularity of an element in $\L$ is preserved under products with non-zero idempotents below the element's support:
\begin{lemma}\label{l:regulatiry-passes-down}
    For non-zero $\lambda\in\L$ and non-zero $\pi, \pi' \in \idempotents$, 
    if $\lambda$
    is $\pi$-regular and $0\neq \pi'\leq \pi$, then $\lambda$ is $\pi'$-regular.
\end{lemma}

\begin{proof}
    Let $\lambda$ be $\pi$-regular and $\kappa$ such that $\kappa\lambda=\pi$.
    Defining $\kappa':=\pi' \kappa$ we have $\kappa'\lambda=\pi '\kappa\lambda=\pi'\pi=\pi'$.
\end{proof}
Fourthly, $\pi$-regular elements in $\L$ are, in a sense, ``locally invertible'':
\begin{lemma}
    \label{l:local-invertibility}
    For non-zero $\pi\in \idempotents$, if $\lambda,\mu\in\L$
    are both $\pi$-regular, then there exists $\kappa\in \pi\L$ so
        that $\kappa\lambda = \pi \mu$. 
\end{lemma}

\begin{proof}
    Let $0\neq \pi \in \idempotents $ and let $\mu,\lambda \in \L$ both be $\pi$-regular.     
    Let $\kappa' $ be such that $\kappa'\lambda = \pi$. Define $\kappa := \kappa'\mu$, so that
    $\kappa\lambda = \kappa'\mu\lambda = \kappa'\lambda\mu=\pi\mu. $
\end{proof}

The following theorem is adapted from \cite[Theorem~5]{kaplansky53}, but we add some further details. 
Classically, given a non-zero functional $\phi$ on a Hilbert space, there trivially exists an element $z\in (\ker \phi)^\perp$ so that $\phi (z) = 1$. Here, some more effort must be expended to establish the analogous step for an $\L$-Hilbert space.


\begin{theorem}[Riesz Representation Theorem]\label{t:RRT}
    Let $H$ be an $\L$-Hilbert space.    
    For every non-zero bounded $\L$-linear functional
    $\varphi \colon H\to\L$, there exists some $f_{\varphi}\in H$ so that,
    for all $h\in H$, $\varphi(h)=\ip h{f_{\varphi}}.$
\end{theorem}

\begin{proof}
    We assume that $\varphi$ is injective on $H$ (if this is not the
    case, by \Cref{c:decomposition}, we may decompose $H=(\ker\varphi)\oplus(\ker\varphi)^{\bot}$
    and restrict $\varphi$ to $(\ker\varphi)^{\bot}$ where $\varphi$
    is injective).

    By Lemma~{\ref{l:L-ban-support-is-realized}}, there exists a normalised $z\in H$ so that
    $\|n_{z}\|=\|z\|=\pi_{z}=\pi_{H}.$

    We claim that $\{z\}^{\bot}=\{\0\}$. Let $w\in\{z\}^{\bot}\subseteq H$,
    but suppose that $w\neq 0$. Since $\varphi$ is injective $\varphi(w)\neq0$
    and hence (by viewing $\L$ as an $\L$-normed space) we obtain $\idempotents\ni \pi_{\varphi(w)}\neq0$.
    Since $\L$-linear maps reduce support (\Cref{l:disjoint-implies-additive}), $0\neq \pi_{\varphi(w)}\leq \pi_{w}\leq \pi_{H}=\pi_{z}.$
   
    By \Cref{l:regularization}, there exists a non-zero component $\pi_{1}$
    of $\pi_{\varphi(w)}$ so that $\varphi(w)$
    is $\pi_{1}$-regular. Therefore 
    $0\neq  \pi_{1}
    =       \pi_{1}\pi_{\varphi(w)}
    \leq    \pi_{\varphi(w)}
    \leq    \pi_{w}
    \leq    \pi_{H}
    =       \pi_{z}$,
    and hence by Lemma \ref{l:support-no-kill} we have, $\pi_{1}z\neq0$.
    By injectivity of $\varphi$, we have $\pi_{1}\varphi(z)=\varphi(\pi_{1}z)\neq0$.
    Again, by \Cref{l:regularization}, there exists a non-zero component
    $\pi_{2}$ of $\pi_{\pi_{1}\varphi(z)}$ ($\leq \pi_{1}$) so that $\pi_{1}\varphi(z)$ is $\pi_{2}$-regular.
    By \Cref{l:regulatiry-passes-down}, since $0\neq \pi_{2}\leq \pi_{1}$, we have that $\varphi(w)$ is also $\pi_{2}$-regular. 
    Now $\pi_{2}$-regularity of both $\pi_{1}\varphi(z)$
    and $\varphi(w)$ with Lemma \ref{l:local-invertibility} implies
    that there exists non-zero $\lambda\in \pi_{2}\L$ so that 
    $\lambda \pi_{2}\varphi(z)= \lambda \pi_{1}\varphi(z) = \pi_{2}\varphi(w)$.
    Therefore 
    $\varphi(\lambda \pi_{2}z-\pi_{2}w)
        =\lambda \pi_{2}\varphi(z)-\pi_{2}\varphi(w)
        =0 
    $.
    By injectivity of $\varphi$, we thus have $\lambda \pi_{2}z-\pi_{2}w=0$.
    Now, since $w\in\{z\}^{\bot}$, we have
    $0
    =\lambda \pi_{2}0
    =\lambda \pi_{2}\ip zw
    =\ip{\lambda \pi_{2}z}{\pi_{2}w}
    =\ip{\pi_{2}w}{\pi_{2}w}
    =\| \pi_{2} w\|.    
    $
    On the other hand, since $0\neq \pi_2 \leq \pi_1 \leq \pi_{\varphi(w)} \leq \pi_w$,
    so that, by \Cref{l:support-no-kill}, $\pi_2 w \neq 0$ implies $ \norm {\pi_2 w }\neq 0$, which is absurd in light of $\| \pi_{2} w\|=0$. 
    We conclude that $w=0$, and hence that $\{z\}^{\bot}=\{0\}$.

    For every $y\in H$, we claim that $y=\ip {y}{z}z$. To this end, let
    $y\in H$ be arbitrary and consider 
    \begin{align*}
        \ip{y-\ip yzz}z 
                        & =\ip yz-\ip yz\ip zz           \\
                        & =\ip y{\pi_{z}z}-\ip yz\|z\|^{2} \\
                        & =\pi_{z}\ip yz-\pi_{z}\ip yz       \\
                        & =0.
    \end{align*}
    Hence $(y-\ip yzz) \perp z$, and because $\{z\}^{\bot}=\{0\}$, we
    have $y-\ip yzz=0$, so that $y=\ip yzz$.

    Finally, define $f_{\varphi}:=\overline{\varphi(z)}z$. Now, for all
    $y\in H$, we obtain
    \[
        \varphi(y)=\varphi(\ip yzz)=\ip yz\varphi(z)=\ip y{\overline{\varphi(z)}z}=\ip y{f_{\varphi}}.\qedhere
    \]
\end{proof}

If $H_1$ and $H_2$ are $\L$-Hilbert spaces and $T \in B(H_1, H_2)$, we define an adjoint of $T$ to be a map $S \in B(H_2, H_1)$ satisfying $\ip{Tx}{y} = \ip{x}{Sy}$ for all $x \in H_1$ and $y \in H_2$.

\begin{theorem}\label{t:adjoint}
Let $H_1$ and $H_2$ be $\L$-Hilbert spaces. Then every $T \in B(H_1, H_2)$ has an adjoint $T^*$. Furthermore the map $T \mapsto T^* \colon B(H_1, H_2) \to B(H_2, H_1)$ is conjugate linear, $\norm{T^*} = \norm{T}$, and $\norm{T^*T} = \norm{T}^2$.
\end{theorem}

\begin{proof}
    Suppose $T \in B(H_1, H_2)$. For $x \in H_2$, the map $y \mapsto \ip{Ty}{x}$ on $H_1$ is bounded, so by \Cref{t:RRT} this map is represented by an element, which we define as $T^*x$. It is clear that $T^* \in B(H_2, H_1)$, the equation $\ip{(\l T)x}{y} = \ip{x}{(\bar{\l} T^*)y}$ shows that $T \mapsto T^*$ is conjugate linear, and 
    \[
    \norm{T} = \sup_{x\in B_{H_1},\ y \in B_{H_2}} |\ip{Tx}{y}| = \sup_{x\in B_{H_1},\ y \in B_{H_2}} |\ip{x}{T^*y}|  = \norm{T^*}. 
    \]
    Finally, note that $\norm{T^*T} \leq \norm{T}^2$ follows from \Cref{p:operators_form_banach_algebra}. Conversely, let $x \in B_{H_1}$, then 
    \[ 
    \norm{Tx}^2 = \ip{Tx}{Tx} = \ip{T^*Tx}{x} \leq \norm{T^*Tx}\norm{x} \leq \norm{T^*T}.
    \]
    Taking the supremum over $x \in B_{H_1}$ yields $\norm{T}^2 \leq \norm{T^*T}$.
\end{proof}

\begin{defn}
    An $\L$-Banach algebra $A$ is called an $\L$-Banach *-algebra if $A$ is equipped with a conjugate $\L$-linear involutive (i.e., $a^{**} = a$ for all $a \in A$) isometric anti-homomorphism $^* \colon A \to A$. An $\L$-Banach *-algebra is called an $\L$-C*-algebra if $\norm{a^*a} = \norm{a}^2$.
\end{defn}
If $H$ is an $\L$-Hilbert space, then \Cref{t:adjoint} shows that $B(H)$ is an $\L$-C*-algebra.





\bibliographystyle{alpha}
\bibliography{L_functional_analysis_arxiv_version}

\end{document}